\documentclass[oneside, a4paper,reqno]{amsart}
\usepackage{amscd,amssymb,amsmath,graphicx,verbatim,mathrsfs}
\usepackage{wasysym}
\usepackage{hyperref}
\usepackage[TS1,OT1,T1]{fontenc}
\usepackage{lscape} 
\usepackage{fullpage} 
\usepackage{pslatex} 
\usepackage{tikz}
\usetikzlibrary{shapes,arrows}
\usepackage[all]{xy}

\newtheorem{theorem}{Theorem}[section]
\newtheorem{lemma}[theorem]{Lemma}
\newtheorem{corollary}[theorem]{Corollary}
\newtheorem{proposition}[theorem]{Proposition}

\newtheorem{lemdef}[theorem]{Lemma and Definition}

\theoremstyle{definition}
\newtheorem{definition}[theorem]{Definition}

\newtheorem{example}[theorem]{Example}

\theoremstyle{remark}
\newtheorem{remark}[theorem]{Remark}

\newcommand{\Xcal}{\ensuremath{\mathcal{X}}}

\newcommand{\Tcal}{\ensuremath{\mathcal{T}}}
\newcommand{\Fcal}{\ensuremath{\mathcal{F}}}
\newcommand{\Ccal}{\ensuremath{\mathcal{C}}}

\newcommand{\Nbb}{\ensuremath{\mathbb{N}}}

\newcommand{\Ucal}{\ensuremath{\mathcal{U}}}
\newcommand{\Vcal}{\ensuremath{\mathcal{V}}}

\newcommand{\Kbb}{\mathbb{K}}
\newcommand{\Zbb}{\mathbb{Z}}

\newcommand{\ra}{\rightarrow}

\numberwithin{equation}{section}

\begin{document}
\title{Silting modules}
\author{Lidia Angeleri H\"ugel, Frederik Marks, Jorge Vit{\'o}ria}
\address{Lidia Angeleri H\"ugel, Dipartimento di Informatica - Settore di Matematica, Universit\`a degli Studi di Verona, Strada le Grazie 15 - Ca' Vignal, I-37134 Verona, Italy} \email{lidia.angeleri@univr.it}
\address{Frederik Marks, Institut f\"ur Algebra und Zahlentheorie, Universit\"at Stuttgart, Pfaffenwaldring 57, 70569 Stuttgart, Germany}
\email{marks@mathematik.uni-stuttgart.de}
\address{Jorge Vit\'oria, Dipartimento di Informatica - Settore di Matematica, Universit\`a degli Studi di Verona, Strada le Grazie 15 - Ca' Vignal, I-37134 Verona, Italy}
\email{jorge.vitoria@univr.it}
\maketitle
\begin{center} {Dedicated to the memory of Dieter Happel} \end{center}
\begin{abstract}
We introduce the new concept of silting modules. These modules generalise  tilting modules over an arbitrary ring, as well as  support $\tau$-tilting modules over a finite dimensional algebra recently introduced by Adachi, Iyama and Reiten.
We show that silting modules generate torsion classes that provide left approximations, and that every partial silting module admits an analogue of the Bongartz complement. Furthermore, we prove that silting modules are in bijection with 2-term silting complexes and with
certain t-structures and co-t-structures in the derived module category. We also see how some of these bijections hold for silting complexes of arbitrary finite length.
\end{abstract}

\section{Introduction}
The concept of tilting is fundamental in representation theory to compare categories of modules or their derived categories. Tilting modules first appeared  in the study of finite dimensional representations of quivers, and they have been generalised in many different ways. Our aim is to bridge together some of these approaches.

Large tilting modules over arbitrary rings  were introduced in \cite{CT}. They play an important role in approximation theory, since they correspond bijectively to the  torsion classes that provide special preenvelopes (see \cite{ATT}).  They also shed a new light on the Homological Conjectures (see for example \cite{AT0,AHT}).
Moreover, these modules are intimately related to localisation, both on the level of module categories and  derived categories. Indeed, over some rings there is an explicit description of all tilting modules using techniques from localisation theory (see for example \cite{AS1,AS2,APST}). 

Silting complexes were first introduced by Keller and Vossieck (\cite{KV}) to study t-structures in the bounded derived category of representations of Dynkin quivers. They generalise tilting complexes - and, thus, finitely  generated tilting modules - in the sense that the associated t-structures yield hearts that are not necessarily derived equivalent to the initial algebra. 
The topic resurfaced recently, in particular through the work of Aihara and Iyama (\cite{AI}),   Keller and Nicol\'as (\cite{KN}), Koenig and Yang (\cite{KY}),  and Mendoza, S\'aenz, Santiago and Souto Salorio (\cite{MSSS}). In \cite{AI}, it was shown that silting complexes over a finite dimensional algebra form a class of objects where mutation can always be performed - contrary to the classical setup of tilting modules.  In fact, mutation requires the existence of exactly two complements for any  almost complete tilting module - a condition which is not always fulfilled, but which can be provided by passing to the class of silting complexes. Furthermore,  in \cite{KN,KY,MSSS}, correspondences relating silting complexes, t-structures and co-t-structures were established,  extending previous work of Hoshino, Kato and Miyachi on 2-term silting complexes and their  associated t-structures and torsion pairs (\cite{HKM}).

Support $\tau$-tilting modules are the module-theoretic counterpart of 2-term silting complexes. They were introduced over finite dimensional algebras by Adachi, Iyama and Reiten (\cite{AIR}), who showed that these modules admit mutation and that there is a mutation-preserving bijection with 2-term silting complexes. In \cite{IJY}, the notion of support $\tau$-tilting was generalised to certain categories of finitely presented functors and correspondences with 2-term silting complexes, t-structures and co-t-structures were established. The goal of this paper is to set up a theory for arbitrary rings and modules that provides a general framework for bijections of this type.

Silting modules over an arbitrary ring are intended to generalise tilting modules in a similar fashion as 2-term silting complexes generalise 2-term tilting complexes and also in the way support $\tau$-tilting modules generalise finitely generated tilting modules over a finite dimensional algebra.
This new class of modules  shares some important features with tilting theory. In particular, the torsion class associated to a silting module provides left approximations, and partial silting modules admit an analogue of the Bongartz complement. 

It turns out that silting modules are related to the class of quasitilting modules studied by Colpi, D'Este and Tonolo in \cite{C,CDT}. As a main feature, these modules  induce   half of the equivalences occurring in Brenner-Butler's classical Tilting Theorem. This forces them to be  finitely generated (\cite{T}). In our work we drop this finiteness condition, and we show that large quasitilting modules can be used to classify some torsion classes (including those generated by silting modules) which provide left approximations.  
Notice that over a finite dimensional algebra,  finitely generated silting and finitely generated quasitilting modules coincide with support $\tau$-tilting modules. Some related results are obtained in parallel work by Wei (\cite{Wei2,Wei3}).

Finally, the proposed concept of silting modules  allows to generalise the correspondences in \cite{AI} and \cite{IJY}. More precisely, we show that    for an arbitrary ring there is a bijection between (not necessarily finitely generated) silting modules and  (not necessarily compact) 2-term silting complexes.  Moreover,  every silting module gives rise to a t-structure which coincides both with the construction due to  Happel, Reiten and Smal{\o} in \cite{HRS} and with the t-structure  studied by Hoshino, Kato and Miyachi in \cite{HKM}. This enables us to prove correspondences between silting modules and certain t-structures and co-t-structures in the unbounded derived category. In fact, these bijections hold for silting complexes of any finite length, thus  extending the correspondences established in \cite{KN,KY,MSSS} to the non-compact setting.

  As mentioned above, to every partial silting module one can associate a silting module obtained by adding a suitable complement, as well as a bireflective subcategory with a corresponding ring epimorphism. In the setting  of support  $\tau$-tilting modules, this was already observed in \cite{J}. Silting theory thus provides an appropriate context for studying ring epimorphisms and localisations. This approach will be explored in a  forthcoming paper.

The structure of the paper is as follows. In section 2 we fix notation, and we recall some important definitions and motivating results for the later sections. Section 3 develops the theory of silting modules. We discuss the existence of silting approximations by passing to the more general notion of a quasitilting module. This allows to classify the  torsion classes   providing left approximations with Ext-projective cokernel (Corollary \ref{air}). 
We present some relevant examples in Section \ref{examples subsection}.  Section 4 is devoted to  silting complexes. In particular, we show how  2-term silting complexes relate with silting modules, t-structures and co-t-structures, generalising  known results established for compact silting complexes (Theorems \ref{bijections general} and \ref{bijections complexes}).

\bigskip

{\bf Acknowledgments.}  The first idea for this paper originates from discussions of the first named author with Dong Yang. 
The authors are also grateful to Qunhua Liu and  Jan \v{S}\v{t}ov\'\i\v{c}ek for further discussion on the topic. 
The first named author is partially supported by  Fondazione Cariparo, Progetto di Eccellenza ASATA.
The second named author is supported by a grant within the DFG-priority program SPP-1489.
The third named author is supported by a Marie Curie Intra-European
Fellowship within the 7th European Community Framework Programme
(PIEF-GA-2012-327376).

\section{Preliminaries}
Throughout, $A$ will be a (unitary) ring, $Mod(A)$ the category of right $A$-modules, and $Proj(A)$ (respectively, $proj(A)$) its subcategory of (finitely generated) projective modules. Modules will always be right $A$-modules. In some contexts, we will be considering algebras $\Lambda$ over an algebraically closed field $\Kbb$.

Morphisms in $Proj(A)$ will be interpreted, without change of notation, both as 2-term complexes concentrated in degrees -1 and 0 in the homotopy category $K(Proj(A))$, and as projective presentations of their cokernels.

The unbounded derived (respectively, homotopy) category of $Mod(A)$ will be denoted by $D(A)$ (respectively, $K(A)$). If we restrict ourselves to bounded or right bounded complexes, we use the usual superscripts $b$ and $-$, respectively.
The term \textit{subcategory} will always refer to a strictly full subcategory. 

For a subcategory $\Xcal$ of $D(A)$ we denote by $\Xcal^{\perp_{>0}}$ the subcategory consisting of the objects $Y$ in $D(A)$ such that $Hom_{D(A)}(X,Y[i])=0$ for all $i>0$ and all $X\in\Xcal$. Similarly, one defines $\Xcal^{\perp_{< 0}}$ and $\Xcal^{\perp_{0}}$. If the subcategory consists of a single object $X$, we write just $X^{\perp_{>0}}$, $X^{\perp_{< 0}}$, and $X^{\perp_0}$. The notation for left orthogonal subcategories is defined analogously. 

For a given $A$-module $M$, we denote by $M^\circ$ the subcategory of $Mod(A)$ consisting of the objects $N$ such that $Hom_A(M,N)=0$, and by $M^{\perp_1}$ the subcategory of $Mod(A)$ consisting of the objects $N$ such that $Ext_A^1(M,N)=0$. 
Further, $Add(M)$ denotes the additive closure of $M$ consisting  of all modules isomorphic to a direct summand of an (arbitrary) direct sum of copies of $M$, while $Gen(M)$ is the subcategory of $M$-generated modules (that is, all epimorphic images of modules in $Add(M)$), and $Pres(M)$ is the subcategory of $M$-presented modules (that is, all modules that admit an $Add(M)$-presentation).

\subsection{Tilting modules}\label{tilting subsection}
Let us begin by recalling some basic facts about (not necessarily finitely generated) tilting modules.

\begin{definition}\label{tilt}
An $A$-module $T$ is said to be \textbf{tilting} if $Gen(T)=T^{\perp_1}$, or equivalently, if $T$ satisfies the following conditions:
\begin{enumerate}
\item[(T1)] the projective dimension of $T$ is less or equal than 1;
\item[(T2)] $Ext^1_A(T,T^{(I)})=0$ for any set $I$;
\item[(T3)] there is an exact sequence
$$\xymatrix{0\ar[r]&A\ar[r]^\phi& T_0\ar[r]& T_1\ar[r]&0}$$
where $T_0$ and $T_1$ lie in $Add(T)$ (and so $\phi$ is a left $Gen(T)$-approximation).
\end{enumerate}
\end{definition}

The subcategory $Gen(T)$ is then called a \textbf{tilting class}. It is a torsion class containing all the injective modules. The notion of partial tilting module is a weakening of this condition. 
There are different definitions in the literature; here we adopt the definition proposed in \cite{CT}.

\begin{definition}\label{ptilt}
We say that an $A$-module $T$ is \textbf{partial tilting} if
\begin{enumerate}
\item[(PT1)] $T^{\perp_1}$ is a torsion class;
\item[(PT2)] $T$ lies in $T^{\perp_1}$.
\end{enumerate}
\end{definition}

Condition (PT1) implies (T1), and it is stronger than (T1) unless $T$ is finitely presented.
Furthermore, once (PT1) is satisfied, (PT2) is equivalent to $Gen(T)$ lying in $T^{\perp_1}$. In fact, also  $Gen(T)$ is then a torsion class, as we are going to see next (cf. \cite[Proposition 4.4]{C}).

\begin{lemma}\label{ext-projectivity}
If an $A$-module $T$ satisfies $Gen(T)\subseteq T^{\perp_1}$, then $(Gen(T),T^\circ)$ is a torsion pair in $Mod(A)$.
\end{lemma}
\begin{proof}
We verify that $Gen(T)={}^\circ(T^\circ)$. Cleary, we have $Gen(T)\subseteq {}^\circ(T^\circ)$. For $M$ in ${}^\circ(T^\circ)$, consider the sequence
$$0\rightarrow \tau_T(M)\rightarrow M \rightarrow M/\tau_T(M)\rightarrow 0$$
given by the trace $\tau_T(M)$ of $T$ in $M$. Applying the functor $Hom_A(T,-)$ to the sequence and using the fact that $Ext^1_A(T,\tau_T(M))=0$, we see that $M/\tau_T(M)$ lies in $T^\circ$. Thus, we have $M\cong \tau_T(M)$, as wanted.
\end{proof}

Recall that a module $M$ in a torsion class $\Tcal$ is  \textbf{Ext-projective} in $\Tcal$ if $Ext^1_A(M,\Tcal)=0$. The condition in the lemma above  can then be rephrased by saying that $T$ is Ext-projective in $Gen(T)$.

\subsection{$\tau$-tilting modules}
Let $\Lambda$ be a finite dimensional $\Kbb$-algebra. We will denote by $\tau$ the Auslander-Reiten translation in the category $mod(\Lambda)$ of finitely generated $\Lambda$-modules. We recall the following definitions from \cite{AIR}.

\begin{definition}
A finitely generated $\Lambda$-module $T$ is said to be 
\begin{itemize}
\item \textbf{$\tau$-rigid} if $Hom_\Lambda(T,\tau T)=0$;
\item \textbf{$\tau$-tilting} if it is $\tau$-rigid and the number of non-isomorphic indecomposable direct summands of $T$ equals the number of isomorphism classes of simple $\Lambda$-modules;
\item \textbf{support $\tau$-tilting} if there is an idempotent element $e$ of $\Lambda$ such that $T$ is a $\tau$-tilting $\Lambda/\Lambda e\Lambda$-module.
\end{itemize}
\end{definition}

In order to generalise these notions to arbitrary rings $A$ and arbitrary $A$-modules, we will need a description that does not use the Auslander-Reiten translation. 

\begin{theorem}\label{tau-free description}
Let $T$ be a finitely generated $\Lambda$-module and $\sigma$ be its minimal projective presentation. 
\begin{enumerate}
\item \cite[Proposition 2.4]{AIR} A $\Lambda$-module $M$ satisfies $Hom_\Lambda(M,\tau T)=0$ if and only if the morphism of abelian groups $Hom_\Lambda(\sigma,M)$ is surjective.
\item \cite[Proposition 5.8]{AS} $T$ is $\tau$-rigid if and only if $Gen(T)\subseteq T^{\perp_1}$.
\item \cite[Corollary 2.13]{AIR} $T$ is support $\tau$-tilting  if and only if
 $Gen(T)$ consists of the $\Lambda$-modules $M$ such that $Hom_\Lambda(\widetilde{\sigma},M)$ is surjective, where $\widetilde{\sigma}$ is the projective presentation of $T$ obtained as the direct sum of $\sigma$ with the complex $(e\Lambda\rightarrow 0)$
 for a suitable idempotent element $e$ of $\Lambda$.
\end{enumerate}
\end{theorem}

\begin{proof}
Statements (1) and (2) follow by similar arguments to the ones used in the given references, using a more general version of the Auslander-Reiten formula (see, for example, \cite{K}). By \cite[Corollary 2.13]{AIR}, $T$ is support $\tau$-tilting if and only if
 $gen(T):=Gen(T)\cap mod(\Lambda)$ consists precisely of the finitely generated $\Lambda$-modules $M$ such that $Hom_\Lambda(\widetilde{\sigma},M)$ is surjective. Consider the torsion pair $(gen(T),T^\circ\cap mod(\Lambda))$ in $mod(\Lambda)$. Note that the  subcategory of $Mod(\Lambda)$ formed by the $\Lambda$-modules $M$ such that $Hom_\Lambda(\widetilde{\sigma},M)$ is surjective forms a torsion class in $Mod(\Lambda)$, whose associated torsion-free class contains $T^\circ\cap mod(\Lambda)$. Moreover, by (2) and Lemma \ref{ext-projectivity}, we also have that $(Gen(T),T^\circ)$ is a torsion pair in $Mod(\Lambda)$. Our claim now follows from the fact that there is a unique torsion pair $(\Tcal,\Fcal)$ in $Mod(\Lambda)$ with $gen(T)\subseteq\Tcal$ and $T^\circ\cap mod(\Lambda)\subseteq\Fcal$, given by the direct limit closure of $(gen(T),T^\circ\cap mod(\Lambda))$ in $Mod(\Lambda)$ (compare \cite{MS}).
 \end{proof}

\subsection{Silting complexes, t-structures and co-t-structures}
Support $\tau$-tilting modules turn out to be in bijection with certain (2-term) complexes, called silting, and they are closely related  with certain t-structures and co-t-structures. Let us  recall some definitions.  First of all, for an object $X$ in $D(A)$, we say that $\{X[i]:i\in\Zbb\}$ {\textbf{generates}} $D(A)$, if whenever a complex $Y$ in $D(A)$  satisfies $Hom_{D(A)}(X[i],Y)=0$ for all $i\in\Zbb$, then $Y=0$.

\begin{definition}\label{small silting} A bounded complex of finitely generated projective $A$-modules $\sigma$ is said to be \textbf{silting} if  
\begin{enumerate}
\item $Hom_{D(A)}(\sigma,\sigma[i])=0$ for all $i>0$;
\item the set $\{\sigma[i]:i\in\Zbb\}$ generates $D(A)$. 
\end{enumerate} 
A silting complex $\sigma$ is said to be \textbf{2-silting} if $\sigma$ is a 2-term complex of projective $A$-modules.
\end{definition}

\begin{remark}\label{generating conditions}
The notion of silting complex has appeared in different references with different \textit{generation} requirements. In order to remove any ambiguity we remark that all these generation properties are equivalent.

Given a ring $A$ and an object $X$ in $D(A)$, the set $\{X[i]:i\in\Zbb\}$ generates $D(A)$ if and only if $D(A)$ is the smallest triangulated subcategory of $D(A)$ which contains $X$ and is closed under coproducts.  In fact, the if-part is clear, and the converse implication follows from \cite[Proposition 4.5]{AJS1} and \cite[Lemma 2.2(1)]{NS}.

If $X$ is compact in $D(A)$, then the two equivalent conditions above are furthermore equivalent to say that $K^b(proj(A))$ is the smallest triangulated subcategory of $D(A)$ containing $X$ and closed under direct summands (i.e., the smallest \textit{thick} subcategory containing $X$). Indeed, under this assumption $\{X[i]:i\in\Zbb\}$ is a generating set for $D(A)$. The converse holds as argued in \cite[Proposition 4.2]{AI}, using arguments of \cite{Ra} and \cite{N}.
\end{remark}

\begin{definition}\cite{BBD, Bo, Pauk}\label{t-def}
Let $D$ be a triangulated category. A \textbf{t-structure} (respectively, a \textbf{co-t-structure}) in $D$ is a pair of  subcategories $(\Vcal^{\leq 0},\Vcal^{\geq 0})$ (respectively, $(\Ucal_{\geq 0},\Ucal_{\leq 0})$) such that
\begin{enumerate}
\item $Hom_D(\Vcal^{\leq 0},\Vcal^{\geq 0}[-1])=0$ (respectively, $Hom_{D(A)}(\Ucal_{\geq 0},\Ucal_{\leq 0}[1])=0$);
\item $\Vcal^{\leq 0}[1]\subseteq \Vcal^{\leq 0}$ (respectively, $\Ucal_{\geq 0}[-1]\subseteq \Ucal_{\geq 0}$);
\item For every $X$ in $D$, there is a triangle
$$\xymatrix{Y\ar[r]&X \ar[r]&W\ar[r] & Y[1]}$$
such that $Y$ lies in $\Vcal^{\leq 0}$ and $W$ lies in $\Vcal^{\geq 0}[-1]$ (respectively, $Y$ lies in $\Ucal_{\geq 0}$ and $W$ lies in $\Ucal_{\leq 0}[1]$).
\end{enumerate}
We use the notations $\Vcal^{\leq n}:=\Vcal^{\leq 0}[-n]$, $\Vcal^{\geq n}:=\Vcal^{\geq 0}[-n]$, $\Ucal_{\geq n}:=\Ucal_{\geq 0}[-n]$ and $\Ucal_{\leq n}:=\Vcal_{\leq 0}[-n]$. A t-structure (respectively, co-t-structure) is furthermore said to be \textbf{bounded} if
$$\bigcup_{n\in\Zbb}\Vcal^{\leq n}=D=\bigcup_{n\in\Zbb}\Vcal^{\geq n}\ \ \ (\text{respectively,\ \  }\bigcup_{n\in\Zbb}\Ucal_{\geq n}=D=\bigcup_{n\in\Zbb}\Ucal_{\leq n}).$$
\end{definition}

For a t-structure $(\Vcal^{\leq 0},\Vcal^{\geq 0})$, the intersection $\Vcal^{\leq 0}\cap\Vcal^{\geq 0}$ is called the \textbf{heart} and $\Vcal^{\leq 0}$ is called the \textbf{aisle}. Note that the aisle completely determines the t-structure since $\Vcal^{\geq 0}=(\Vcal^{\leq 0})^{\perp_0}[1]$. Furthermore, for a t-structure, the triangles in axiom (3) are functorial, giving rise to \textbf{truncation functors} (see \cite{BBD}).

\begin{example}\label{example t-str}
(1) The pair $(D^{\leq 0},D^{\geq 0})$ in $D(A)$, where $D^{\leq 0}$ (respectively, $D^{\geq 0}$) is the subcategory of complexes with cohomologies lying in non-positive (respectively, non-negative) degrees, is a t-structure, called the \textbf{standard t-structure}. We denote its associated truncation functors by $\tau^{\leq n}$ and $\tau^{\geq n}$, for all $n\in\Zbb$.

\smallskip

(2) \cite{Bo, Pauk} Consider the triangulated subcategory $K_p(A)$ of $K(A)$ of homotopically projective complexes. The canonical functor from $K(A)$ to $D(A)$ is known to induce a triangle equivalence between $K_p(A)$ and $D(A)$ (see, for example, \cite{Ke}). We use this fact throughout without further mention. The pair $(K_{\geq 0},K_{\leq 0})$ in $K_p(A)$, where $K^{\geq 0}$ (respectively, $K_{\leq 0}$) is the subcategory of complexes whose negative (respectively, positive) components are zero, is a co-t-structure, called the \textbf{standard co-t-structure}. The triangles in axiom (3) can be obtained (non-functorially) using the so-called \textbf{stupid truncations}, where zero replaces the components of the complex which are outside the required bound.

\smallskip

(3) \cite[Theorem 2.1]{HRS} A torsion pair $(\Tcal,\Fcal)$ in $Mod(A)$ induces  a t-structure $(D_\Tcal^{\leq 0},D_\Fcal^{\geq 0})$  in $D(A)$ given by
$$D_\Tcal^{\leq 0}:=\{X\in D(A): H^0(X)\in\Tcal, H^i(X)=0, \forall i>0\}$$
$$D_\Fcal^{\geq 0}:=\{X\in D(A): H^{-1}(X)\in\Fcal, H^i(X)=0, \forall i<-1\}.$$

\smallskip
 
(4) \cite[Proposition 3.2]{AJS} For every object $X$ in $D(A)$ there is a t-structure $(aisle(X), X^{\perp_{<0}})$, called \textbf{the t-structure generated by} $X$, where $aisle(X)$ is the smallest coproduct-closed suspended subcategory of $D(A)$ containing $X$. Recall that an additive subcategory of $D(A)$ is called \textbf{suspended}, if it is closed under extensions and positive shifts.
\end{example}

The following theorems, which we will generalise to a larger context, relate some of the concepts introduced above.
For details on the notion of a silting t-structure we refer to \cite[Definition 4.9]{AI} and Definition \ref{silting}. For related results, see also \cite{IJY,MSSS}.

\begin{theorem}\cite[Theorem 6.1]{KY}\cite{KN}\label{KY correspondence}
Let $\Lambda$ be a finite dimensional $\Kbb$-algebra. There are bijections between
\begin{enumerate}
\item isomorphism classes of basic silting complexes in $K^b(proj(\Lambda))$;
\item bounded t-structures in $D^b(mod(\Lambda))$ whose heart is equivalent to $mod(\Gamma)$ for some $\Kbb$-algebra $\Gamma$;
\item bounded co-t-structures in $K^b(proj(\Lambda))$.
\end{enumerate}
\end{theorem}

\begin{theorem}\cite[Theorem 3.2]{AIR},\cite[Theorem 4.10]{AI}\label{bij}
Let $\Lambda$ be a finite dimensional $\Kbb$-algebra. There are bijections between
\begin{enumerate}
\item isomorphism  classes of basic support $\tau$-tilting $\Lambda$-modules;
\item isomorphism  classes of basic 2-silting complexes in $K^b(proj(\Lambda))$;
\item 2-silting t-structures $(\Ucal^{\leq 0},\Ucal^{\geq 0})$ in $D(\Lambda)$.
\end{enumerate}
\end{theorem}

\section{Silting modules}\label{section silting}

We want to introduce a class of modules that generalises tilting modules over arbitrary rings, and at the same time, coincides with support $\tau$-tilting modules when restricting to finitely generated modules over a finite dimensional algebra. One of the main common features of tilting and support $\tau$-tilting modules is their connection to torsion classes that provide left (and right) approximations. We therefore start by discussing the existence of such approximations. Afterwards, we define silting modules and study further properties.

\subsection{Approximations and quasitilting modules}

A crucial feature of tilting theory is that tilting classes provide special preenvelopes. 
Recall that, given  
 a subcategory $\Tcal$ of $Mod(A)$, a  \textbf{special $\Tcal$-preenvelope} of an $A$-module $M$ is a short exact sequence $$0\longrightarrow M\stackrel{\phi}{\longrightarrow} B\longrightarrow C\longrightarrow 0$$ such that $B$ lies in $\Tcal$ and   $Ext_A^1(C, \Tcal)=0$ 
(and so $\phi$ is a left $\Tcal$-approximation of $M$).

\begin{theorem}\cite[Theorem 2.1]{ATT}\label{tilting classes}
 A torsion class $\Tcal$ in $Mod(A)$ is a tilting torsion class if and only if every $A$-module admits a  special $\Tcal$-preenvelope.
\end{theorem}

Also support $\tau$-tilting modules induce  approximation sequences,  but the map $\phi$ is not injective  in general.  
So, we now turn to torsion classes  providing left  approximations with Ext-projective cokernel. The classification of such torsion classes will lead us 
to the notion of a quasitilting module, and it will allow to recover a result from \cite{AIR} relating support $\tau$-tilting modules with functorially finite torsion classes (see Remark \ref{functfin}).

First, we recall the notion of a $\ast$-module (\cite{C}). Such modules arise in the literature as  capturing \textit{half} of the categorical equivalences of the Brenner-Butler theorem in tilting theory. In fact $\ast$-modules are precisely those $A$-modules $T$ such that the functor $Hom_A(T,-)$ induces an equivalence between $Gen(T_A)$ and $Cogen(\mathsf{D}(T)_{B})$, where $B=End_A(T)$ and $\mathsf{D}(T)$ is the dual of $T$ with respect to an injective cogenerator of $Mod(A)$. This forces them to be finitely generated (\cite{T}). For our purpose we have to drop this finiteness condition and work with the following ``large version'' of the notion of a $\ast$-module.

\begin{definition}
An $A$-module $T$ is a \textbf{$\ast$-module} if $Gen(T)= Pres(T)$, and $Hom_A(T,-)$ is exact for short exact sequences in $Gen(T)$.
\end{definition}

Quasitilting modules were introduced in  \cite{CDT} as the (self-small) $\ast$-modules $T$ for which $Gen(T)$ is a torsion class. In fact, there are many equivalent ways of defining such modules, cf.~\cite[Proposition 2.1]{CDT}. 
For a subcategory $\Ccal$ of $Mod(A)$, we denote by $\overline{\Ccal}$ the subcategory formed by the submodules of all modules in $\Ccal$.

\begin{lemdef}\label{eq def quasitilting}
The following statements are equivalent for an $A$-module $T$.
\begin{enumerate}
\item $T$ is a $\ast$-module and $Gen(T)$ is a torsion class;
\item $Pres(T)=Gen(T)$ and $T$ is Ext-projective in $Gen(T)$;
\item $Gen(T)=\overline{Gen(T)}\cap T^{\perp_1}$.
\end{enumerate}
We say that  $T$ is \textbf{quasitilting} if it satisfies any of the equivalent conditions above.
\end{lemdef}
\begin{proof}
(1)$\Rightarrow$(2): We only have to show that $Ext^1_A(T,Gen(T))=0$. Consider a short exact sequence in $Mod(A)$ $$\xymatrix{0\ar[r]&M\ar[r]&N\ar[r]^g&T\ar[r]&0}$$
with $M$ in $Gen(T)$. Since $Gen(T)$ is a torsion class, $N$ lies in $Gen(T)$ and, by assumption, the  sequence remains exact when applying the functor $Hom_A(T,-)$. But then it is split exact as $1_T$ factors through $g$. 

(2)$\Rightarrow$(1): It is clear that if $T$ is Ext-projective in $Gen(T)$ then $Hom_A(T,-)$ is exact for short exact sequences in $Gen(T)$. By Lemma \ref{ext-projectivity}, $Gen(T)$ is a torsion class and, thus, we have (1).

(2)$\Rightarrow$(3): It is clear that $Gen(T)\subseteq \overline{Gen(T)}\cap T^{\perp_1}$. For the reverse inclusion, let $N$ lie in $\overline{Gen(T)}\cap T^{\perp_1}$ and let $M$ be an object in $Gen(T)$ such that there is a monomorphism $f:N\rightarrow M$. Clearly, $C:=Coker(f)$ lies in $Gen(T)$ and, thus, in $Pres(T)$. So there is a surjection $g:T^\prime\rightarrow C$   with $T^\prime$ in $Add(T)$ such that $K:=Ker(g)$ lies in $Gen(T)$. Since $Ext^1_A(T^\prime,N)=0$ by assumption, we obtain the following commutative diagram of short exact sequences:
$$\xymatrix{0\ar[r]& K\ar[r]\ar[d]^a&T^\prime\ar[d]^b\ar[r]^g&C\ar[r]\ar @{=}[d]&0\\0\ar[r]& N\ar[r]^f&M\ar[r]&C\ar[r]&0.}$$ 
Now, the snake lemma shows that $Coker(a)=Coker(b)$ and, thus, $Coker(a)$ lies in $Gen(T)$. Since $Gen(T)$ is extension-closed by Lemma \ref{ext-projectivity}, we conclude that $N$ lies in $Gen(T)$.

(3)$\Rightarrow$(2) We only need to show that $Gen(T)\subseteq Pres(T)$. Let $M$ lie in $Gen(T)$ and consider the universal map $u:T^{(I)}\rightarrow M$, where $I=Hom_A(T,M)$. Clearly $u$ is surjective, and since $Ext_A^1(T,T^{(I)})=0$, it is easy to see that $Ker(u)$ lies in $T^{\perp_1}$.
 By assumption $Ker(u)$ then lies in $Gen(T)$, so $M$ lies in $Pres(T)$.
\end{proof}

We will  require our modules to be \textbf{finendo}, i.e. finitely generated over their endomorphism ring, as this characterises the modules $T$ for which $Gen(T)$ provides left approximations (\cite[Proposition 1.2]{ATT}). Note that this is further equivalent to $Gen(T)$ being closed for direct products (\cite[Lemma on p.408]{CM}). Recall that a module is called {\textbf{faithful}}, if its annihilator is zero. The following lemma extends a result relating $\ast$-modules to tilting.

\begin{lemma}\label{finendo star}(cf.~\cite[Corollary 2]{DH}, \cite[Corollary 6]{C2}) 
An $A$-module $T$ is a finendo $\ast$-module if and only if it is a tilting $A/Ann(T)$-module.
\end{lemma}
\begin{proof}
Set $\bar{A}=A/Ann(T)$. Since $Ann(T)=Ann(Gen(T))$, it follows that $Gen(T_A)=Gen(T_{\bar{A}})$. Therefore, it is easy to see that $T$ is a finendo $\ast$-module over $A$ if and only if $T$ is a finendo $\ast$-module over $\bar{A}$. So, without loss of generality, it is enough to show that $T$ is a faithful finendo $\ast$-module over $A$ if and only if $T$ is a tilting $A$-module.

The if-part is clear. For the only-if-part, consider a faithful finendo $\ast$-module $T$. 
As in \cite[Theorem 3]{C2}, 
we see that 
all injective $A$-modules are contained in $Gen(T)$, and
$Gen(T)\subseteq T^{\perp_1}$. 
We repeat the arguments for the reader's convenience. Since $T$ is faithful there is a monomorphism $\phi:A\rightarrow T^\alpha$ for some set $\alpha$, where $T^\alpha$ lies in $Gen(T)$ as $T$ is finendo. Now every surjection $A^{(I)}\rightarrow E$ to an injective module $E$ extends to a surjection $(T^\alpha)^{(I)}\rightarrow E$, showing the first claim.
Further, given $M$  in $Gen(T)$, the functor $Hom_A(T,-)$ is exact on the short exact sequence in $Gen(T)$ induced by an injective envelope $M\rightarrow E(M)$ and, since $Ext^1_A(T,E(M))=0$, we get $Ext_A^1(T,M)=0$. 

Now, by Lemma \ref{ext-projectivity}, we have that $Gen(T)$ is a torsion class. Thus, by Lemma and Definition \ref{eq def quasitilting}, $T$ is a quasitilting module and $Gen(T)=\overline{Gen(T)}\cap T^{\perp_1}$. But $\overline{Gen(T)}=Mod(A)$ since every injective module lies in $Gen(T)$, and $Gen(T)=T^{\perp_1}$ as wanted.
\end{proof}

Let us turn to the existence of approximations.

\begin{proposition}\label{bar}
The following are equivalent for an $A$-module $T$.
\begin{enumerate}
\item $T$ is a finendo quasitilting module.
\item $T$  is Ext-projective in $Gen(T)$ and
there is an exact sequence $$\xymatrix{A\ar[r]^{\phi} & T_0\ar[r] & T_1\ar[r] & 0,}$$ with $T_0$ and $T_1$ in $Add(T)$ and $\phi$ a left $Gen(T)$-approximation.
\end{enumerate}
\end{proposition}
\begin{proof}

(1)$\Rightarrow$(2):   $T$ is by definition Ext-projective in $Gen(T)$ and, moreover, $T$ is a tilting $\bar{A}$-module by Lemma \ref{finendo star}. Then there is a short 
exact sequence $$\xymatrix{0\ar[r]&\bar{A}\ar[r]^{\bar{\phi}} & T_0\ar[r] & T_1\ar[r]&0}$$  with $T_0$ and $T_1$ in $Add(T)$ and  $\bar{\phi}$ a left $Gen(T_{\bar{A}})$-approximation in $Mod(\bar{A})$. The composition with the canonical projection $\pi:A\rightarrow \bar{A}$ then yields the desired left $Gen(T)$-approximation $\phi=\bar{\phi}\pi:A\to T_0$  in $Mod(A)$.

(2)$\Rightarrow$(1): We  have to show that $T$ is an $\bar{A}$-tilting module.
First, we see that $Gen(T)$  is contained  $Ker(Ext^1_{\bar{A}}(T,-))$. Indeed, every short exact sequence 
$0\rightarrow M\rightarrow N\rightarrow T\rightarrow 0$ in $Mod(\bar{A})$,
with $M$ in $Gen(T)$, splits in $Mod(A)$ 
as $T$ is Ext-projective, and thus it splits in $Mod(\bar{A})$. Now, we show that $Ann(T)= Ker(\phi)$. In fact,  $Ann(T)\subseteq Ker(\phi)$ as $T_0$ lies in $Gen(T)$. For the reverse inclusion
note that $Ann(T)$ is the intersection of the kernels of all maps in $Hom_A(A,T)$. Since every map $f:A\rightarrow T$ factors through $\phi$, we infer $Ker(\phi)\subseteq Ker(f)$.

Therefore, $\phi$ factors as  $\phi=\bar{\phi}\pi$  through   the canonical projection $\pi:A\rightarrow \bar{A}$. From the short exact sequence $$\xymatrix{0\ar[r]&\bar{A}\ar[r]^{\bar{\phi}} & T_0\ar[r] & T_1\ar[r]&0}$$ we deduce that every module 
$X$ in $Ker(Ext^1_{\bar{A}}(T,-))$, being generated by $\bar{A}$ and satisfying  $Ext^1_{\bar{A}}(T_1,X)=0$, is also generated by $T_0$, and thus by $T$.
Hence  $Gen(T)= Ker(Ext^1_{\bar{A}}(T,-))$, and the proof is complete.
\end{proof}

We can now classify the torsion classes that yield left approximations with Ext-projective cokernel.

\begin{theorem}\label{ATTanalog}
The following are equivalent for a torsion class $\Tcal$ in $Mod(A)$.
\begin{enumerate}
\item For every $A$-module $M$ there is a sequence
$$\xymatrix{M\ar[r]^{\phi} & B\ar[r] & C\ar[r]&0}$$
such that $\phi$ is a left $\Tcal$-approximation and $C$ is Ext-projective in $\Tcal$.
\item There is a finendo quasitilting $A$-module $T$ such that $\Tcal=Gen(T)$.
\end{enumerate}
\end{theorem}
\begin{proof}
(1)$\Rightarrow$(2): Choose $M=A$ with an approximation sequence $$\xymatrix{A\ar[r]^{\phi} & B\ar[r] & C\ar[r]&0}$$ and set $T=B\oplus C$. Clearly, we have $Gen(T)\subseteq \Tcal$. Conversely, if $X$ is a module in $\Tcal$, any surjection $f:A^{(I)}\ra X$ factors through the $\Tcal$-approximation $\phi^{(I)}$ via a surjection $B^{(I)}\ra X$, showing that $X$ lies in $Gen(T)$. Thus, we have that $Gen(T)=\Tcal$. By Proposition \ref{bar}, it remains to show that $T$ is Ext-projective in $Gen(T)$. In fact, by assumption, we have to verify this only for $B$.
As in the proof of Proposition \ref{bar} we obtain a short exact sequence $$\xymatrix{0\ar[r]&\bar{A}\ar[r]^{\bar{\phi}} & B\ar[r] & C\ar[r]&0}$$
 over $\bar{A}=A/Ann(T)$, and we see that  $Gen(T)$  is contained  $Ker(Ext^1_{\bar{A}}(C,-))$.
Using the projectivity of $\bar{A}_{\bar{A}}$, we infer that $Gen(T)$  is  also contained  $Ker(Ext^1_{\bar{A}}(B,-))$. Consider now a 
short exact sequence 
$0\rightarrow M\rightarrow N\rightarrow B\rightarrow 0$ in $Mod(A)$
with $M$ in $Gen(T)$. Since $Gen(T)$ is a torsion class, also $N$ belongs to $Gen(T)$ and the sequence actually lies in  $Mod(\bar{A})$. Then it splits in  $Mod(\bar{A})$, and thus it also splits in $Mod({A})$. So $B$ is Ext-projective in $Gen(T)$.

(2)$\Rightarrow$(1): As in \cite[Proposition 1.2]{ATT}, we use the approximation sequence for $A$ in Proposition \ref{bar} to construct approximation sequences for all $A$-modules $M$, where the cokernels turn out to lie in  $Add(T)$ and thus  are Ext-projective modules in $\Tcal$.\end{proof}

The following lemma tells how to recover a quasitilting module from its associated torsion class.

\begin{lemma}\label{recovering}
If $T$ is a quasitilting module, then $Add(T)$ is the class of  Ext-projective modules in $Gen(T)$.
\end{lemma}
\begin{proof}
If $T$ is Ext-projective in $Gen(T)$, then so is every module in $Add(T)$. Conversely, given an Ext-projective module $M$ in $Gen(T)=Pres(T)$, there is a surjection $f:T^{\prime}\rightarrow M$, for some $T^\prime$ in $Add(T)$, with $Ker(f)$ in $Gen(T)$. The Ext-projectivity of $M$ implies that the short exact sequence induced by $f$ splits and, thus $M$ lies in $Add(T)$.
\end{proof}

Consequently, two quasitilting modules have the same additive closure if and only if they generate the same torsion class.
We will thus say that two quasitilting  modules $T_1$ and $T_2$ are \textbf{equivalent} if $Add(T_1)=Add(T_2)$. Theorem \ref{ATTanalog} can now be rephrased as follows.

\begin{corollary}\label{air}
 There is a bijection between equivalence classes of finendo quasitilting $A$-modules and torsion classes $\Tcal$ in $Mod(A)$ such that every $A$-module has a left $\Tcal$-approximation with Ext-projective cokernel. \end{corollary}

\subsection{Silting modules}\label{subsection silting}
In this subsection we study (partial) silting modules, the main objects under consideration in this work.
These modules will be defined in a way suggested by Theorem \ref{tau-free description}. For a morphism $\sigma$ in $Proj(A)$, we consider the class of $A$-modules
$$\mathscr{D}_\sigma:=\{X\in Mod(A)|Hom_A(\sigma,X)\ \text{is surjective}\}.$$ 
 We collect some useful properties of $\mathscr{D}_\sigma$.

\begin{lemma}\label{ext-orthogonality1}
Let $\sigma$ be a map in $Proj(A)$ with cokernel $T$.
\begin{enumerate}
\item $\mathscr{D}_\sigma$ is closed under epimorphic images,  extensions, and direct products.
\item The class $\mathscr{D}_\sigma$ is contained in $T^{\perp_1}$.
\item An $A$-module $X$ belongs to $\mathscr{D}_\sigma$ if and only if for some (respectively, all) projective presentation(s) $\omega$ of $X$ the condition $Hom_{D(A)}(\sigma,\omega[1])=0$ is satisfied.
\end{enumerate}
\end{lemma}
\begin{proof}
The proof of statement (1) is left to the reader.

(2) Set $\sigma: P_{-1}\ra P_0$ and write $\sigma=i\pi$ with  $\pi:P_{-1}\rightarrow Im(\sigma)$ and $i:Im(\sigma)\rightarrow P_0$. By applying the functor $Hom_A(-,N)$, with $N$ in $\mathscr{D}_\sigma$, to the short exact sequence induced by the monomorphism $i:Im(\sigma)\rightarrow P_0$
we get the exact sequence
$$\xymatrix{ Hom_A(P_0,N)\ar[r]^{\!\!\!\!\!i_*} & Hom_A(Im({\sigma}),N)\ar[r] & Ext_A^1(T,N)\ar[r] & 0.}$$
We show that $i_*$ is surjective. Consider a test map $f:Im(\sigma)\ra N$.
Since $N$ belongs to $\mathscr{D}_\sigma$, there is a map $g:P_0\ra N$ such that $f\pi=g i\pi$. Consequently, since $\pi$ is an epimorphism, we get $f=g i$, as wanted.

(3) This is an easy observation, based on \cite[Lemma 3.4]{AIR}. 
\end{proof}

\begin{definition}\label{def p silting}
We say that an $A$-module $T$ is
\begin{itemize}
\item  \textbf{partial silting} if there is a projective presentation $\sigma$ of $T$ such that 
\begin{enumerate}
\item[(S1)] $\mathscr{D}_\sigma$ is a torsion class.
\item[(S2)] $T$ lies in $\mathscr{D}_\sigma$.
\end{enumerate}
\item \textbf{silting} if there is a projective presentation $\sigma$ of $T$ such that $Gen(T)=\mathscr{D}_\sigma$.
\end{itemize}
We will then say that $T$ is (partial) silting \textbf{with respect to} $\sigma$.
\end{definition}

\begin{remark}\label{a few observations}
(1) If $T$ is partial silting, then $Gen(T)\subseteq \mathscr{D}_\sigma \subseteq T^{\perp_1}$  by Lemma \ref{ext-orthogonality1}(2), and   $(Gen(T),T^\circ)$ is a torsion pair by Lemma \ref{ext-projectivity}. The same arguments show that every silting module is partial silting.

\smallskip

(2) Since $\mathscr{D}_\sigma$ is always closed for epimorphic images and extensions, condition (S1) is equivalent to require that $\mathscr{D}_\sigma$ is closed for coproducts. This is always true when $\sigma$ is a map in $proj(A)$ and, thus, a compact object in $D(A)$. So, in this case, $T$ is partial silting if and only if  $Hom_{D(A)}(\sigma,\sigma[1])=0$. The latter property hints on the choice of  the name \textit{silting} for our modules, which will indeed be justified by the relation with  (2-term) silting complexes (to be explored in section \ref{Silting complexes}).

Notice, however, that in general $\mathscr{D}_\sigma$ can contain  $T$, and even all direct sums of copies of $T$, without being a torsion class. For example,  the generic module $G$ over the Kronecker algebra (the path algebra of the quiver $\xy\xymatrixcolsep{2pc}\xymatrix{ \bullet \ar@<0.5ex>[r]
  \ar@<-0.5ex>[r] & \bullet } \endxy$) satisfies conditions (T1) and (T2) in Definition \ref{tilt}. Taking  a monomorphic presentation  $\sigma$ of $G$, we obtain a class  $\mathscr{D}_\sigma=G^{\perp_1}$  containing  $Gen(G)$. But $\mathscr{D}_\sigma$ is not a torsion class (and $G$ is not partial tilting according to Definition \ref{ptilt}), because it  is not closed under direct sums. Indeed, every adic module $S_{-\infty}$ belongs to $G^{\perp_1}$, while $S_{-\infty}\,^{(\omega)}$ does not. This follows from   \cite[Proposition 1 and Remark on p.265]{O} 
 stating that a torsion-free regular module belongs to $G^{\perp_1}$ if and only if it is pure-injective. For details on infinite dimensional modules over hereditary algebras we refer to  \cite{R,RR}.
 
 \smallskip

(3) Note that the definitions in \ref{def p silting} depend on the choice of $\sigma$: not all projective presentations of a silting or partial silting module will fulfill conditions (S1) and (S2). Further, $T$ can be partial silting with respect to different projective presentations giving rise to different associated torsion classes. However, there is a unique torsion class, $Gen(T)$, which can turn a module $T$ into a silting module.

\end{remark}

There is an evident parallel between (S1) and (S2) and the axioms (PT1) and (PT2) defining partial tilting modules and, thus, also with (T1) and (T2) in the definition of a tilting module. We will later obtain an analogue of (T3) in Theorem \ref{eq def silting module}. Moreover, the definition of silting clearly resembles the condition $Gen(T)=T^{\perp_1}$  defining tilting. 
Let us make this comparison more precise. Recall that an $A$-module $T$ is said to be \textbf{sincere} if $Hom_A(P,T)\not= 0$ for all non-zero projective $A$-modules $P$.

\begin{proposition}\label{tilting}
\begin{enumerate}
\item An $A$-module $T$ is (partial) tilting if and only if $T$ is a (partial) silting module with respect to a monomorphic projective presentation.
\item A module $T$ of projective dimension at most one is tilting if and only if it is a  sincere silting module.
\end{enumerate}
\end{proposition}
\begin{proof}
(1) If $T$ is a partial tilting module, there is a monomorphic projective presentation $\sigma$ of $T$,
and $\mathscr{D}_\sigma=T^{\perp_1}$. 
Since $Ext_A^1(T,T)=0$, $T$ lies in $\mathscr{D}_\sigma$, so that $T$ is partial silting with respect to  $\sigma$. If, furthermore, $T$ is tilting, then $Gen(T)=T^{\perp_1}=\mathscr{D}_\sigma$, thus showing that $T$ is silting. The converse implication is shown similarly.

(2) If $T$ is tilting, then it is a faithful module and, therefore, sincere. Conversely, assume that $T$ is a sincere silting module with respect to a projective presentation $\sigma:P_{-1}\ra P_0$.  Since $T$ has projective dimension at most one, $Im(\sigma)$ is a projective $A$-module and  $Ker(\sigma)$ is a direct summand of  $P_{-1}$. But then, as  $T$ lies in $\mathscr{D}_\sigma$ and every morphism $P_{-1}\ra T$ factors through $\sigma$, we have $Hom_A(Ker(\sigma),T)=0$. Since $Ker(\sigma)$ is projective and $T$ is sincere, it follows that $Ker(\sigma)=0$ and $T$ is tilting by (1).
\end{proof}

Notice that even if a module has projective dimension one, it can happen that monomorphic presentations are not the  ones to consider for verifying the silting condition. So not all silting modules of projective dimension 1 are tilting, as illustrated 
in Subsection \ref{examples subsection}. The next proposition relates silting modules to quasitilting modules.

\begin{proposition}\label{quasi}
\begin{enumerate}
\item All silting modules are finendo quasitilting.
\item A module is tilting if and only if it is  faithful silting (and if and only if it is  faithful finendo quasitilting).
\end{enumerate}
\end{proposition}
\begin{proof}
(1) Let $T$ be silting with respect to a projective presentation $\sigma:P_{-1}\rightarrow P_0$. Then we know from Lemma \ref{ext-orthogonality1}(1) that $Gen(T)=\mathscr{D}_\sigma$ is closed under direct products, which means that $T$ is finendo.
Further, $T$ is Ext-projective in $Gen(T)$ by Remark \ref{a few observations}(1). It remains to show that $Gen(T)\subseteq Pres(T)$. Let $M$ lie in $Gen(T)$, let $I$ be $Hom_A(T,M)$, and consider the universal map $u:T^{(I)}\rightarrow M$ (which is then surjective). We will show that $K:=Ker(u)$ lies in $\mathscr{D}_\sigma=Gen(T)$, thus finishing the proof. 

Pick $f:P_{-1}\to K$. Since $T^{(I)}$ lies in $\mathscr{D}_\sigma$, we have the following commutative diagram
$$\xymatrix{& P_{-1}\ar[d]^f\ar[r]^\sigma&P_0\ar[r]^\pi\ar[d]^g &T\ar[r]\ar[d]^h&0\\ 0\ar[r]& K\ar[r]^k&T^{(I)}\ar[r]^u&M\ar[r]&0.}$$
By the universality of $u$, there is  $\tilde{h}:T\rightarrow T^{(I)}$ such that $u\tilde{h}=h$. It then follows by a routine diagram chase that there is a map $\tilde{g}:P_0\rightarrow K$ such that $\tilde{g}\sigma=f$, as wanted.

Statement (2) is an immediate consequence of Lemma \ref{finendo star}.
\end{proof}

In particular, it follows that we can recover the additive closure of a silting module from its associated torsion class (see Lemma \ref{recovering}). We will say that two silting  modules $T$ and $T^\prime$ are \textbf{equivalent} if $Add(T)=Add(T^\prime)$.

The next result measures the difference between silting and quasitilting modules, and it characterises silting modules in terms of a condition (S3) which is the silting counterpart of condition (T3) in Definition \ref{tilt}.

\begin{proposition}\label{eq def silting module}
The following are equivalent for an $A$-module $T$ and a projective presentation $\sigma$ of $T$.
\begin{enumerate}
\item $T$ is a silting module with respect to $\sigma$. 
\item $T$ is a partial silting module with respect to $\sigma$ and 
\begin{enumerate}
\item[(S3)] there is an exact sequence $$\xymatrix{A\ar[r]^{\phi} & T_0\ar[r] & T_1\ar[r] & 0,}$$ with $T_0$ and $T_1$ in $Add(T)$ and $\phi$ a left $\mathscr{D}_\sigma$-approximation.
\end{enumerate}
\end{enumerate}
\end{proposition}
\begin{proof}
(1)$\Rightarrow$(2): This follows from Proposition \ref{quasi}(1) and Proposition \ref{bar} using that $\mathscr{D}_\sigma=Gen(T)$.

(2)$\Rightarrow$(1): Since $T$ is partial silting with respect to $\sigma$, it is clear that $Gen(T)\subseteq \mathscr{D}_\sigma$. If $M$ lies in $\mathscr{D}_\sigma$, any surjection $f:A^{(I)}\ra M$
factors through the $\mathscr{D}_\sigma$-approximation $\phi^{(I)}$ via a surjection $g:T_0^{(I)}\ra M$. Thus, $M$ lies in $Gen(T)$.
\end{proof}

A well-known result of Bongartz - later proved in full generality in \cite{CT} - states that every partial tilting module can be completed to a tilting module. The following theorem now generalises it to our setting.

\begin{theorem}\label{complement}
Every partial silting $A$-module $T$ with respect to a projective presentation $\sigma$ is a direct summand of a silting $A$-module $\bar{T}=T\oplus M$ with the same associated torsion class, that is, $Gen(\bar{T})=\mathscr{D}_\sigma$. 
\end{theorem}
\begin{proof}
Let $T$ be a partial silting $A$-module and let $\sigma:P_{-1}\rightarrow P_0$ be a projective presentation of $T$. In order to find a complement for $T$, we begin by constructing an approximation sequence for $A$ in $\mathscr{D}_\sigma$. Consider the universal map $\psi:P_{-1}\,^{(I)}\ra A$ with $I=Hom_A(P_{-1},A)$. We get the following pushout diagram
\begin{equation}\label{diag1}
\xymatrix{P_{-1}\,^{(I)}\ar[d]^{\psi}\ar[r]^{\sigma^{(I)}} & P_0\,^{(I)}\ar[d]^{\psi_1}\ar[r] & T^{(I)}\ar @{=}[d]\ar[r] & 0\\
A\ar[r]^{\phi} & M\ar[r]^{\pi} & T^{(I)}\ar[r] & 0.}
\end{equation}
If $M$ lies in $\mathscr{D}_\sigma$ then it easily follows from the universal property of the pushout that $\phi$ is a left $\mathscr{D}_\sigma$-approximation. We will, therefore, show that any map $g:P_{-1}\ra M$ factors through $\sigma$.
Since $T^{(I)}$ lies in $\mathscr{D}_\sigma$, the composition $\pi g$ must factor through $\sigma$ via some map $g_1:P_0\ra T^{(I)}$, yielding the following commutative diagram
$$\xymatrix{P_{-1}\ar[r]^\sigma\ar[d]^g & P_0\ar[d]^{g_1}\\ M\ar[r]^\pi & T^{(I)}}$$
Moreover, since $P_0$ is projective, there is a map $g_2:P_0\ra M$ such that $g_1=\pi g_2$. It follows from a routine diagram chase that $g_2\sigma-g$ factors through $\phi$. 
Now, by the construction of $\psi$ and the commutativity of diagram (\ref{diag1}), there are component maps $\psi^\prime:P_{-1}\ra X$ and $\psi_1^\prime:P_0\ra M$ fulfilling 
$g_2\sigma-g=f\psi^\prime=\psi_1^\prime\sigma.$
Consequently, the map $g$ factors through $\sigma$, proving that $M$ lies in $\mathscr{D}_\sigma$.

We will now prove that $\bar{T}:=T\oplus M$ is a silting $A$-module. Since the left square of diagram (\ref{diag1}) is a pushout diagram, it yields a projective presentation of $M$
$$\xymatrix{P_{-1}\,^{(I)}\ar[r]^{\!\!\!\!(g\,\,\,\sigma^{(I)})} & A\oplus P_0\,^{(I)}\ar[r]^{\,\,\,\,\,\,\Tiny{\left(\begin{array}{c}-\phi\\f\end{array}\right)}} & M\ar[r] & 0.}$$
This gives us a projective presentation of $\bar{T}$ by considering the direct sum $\gamma:=\sigma\oplus (g\,\,\,\sigma^{(I)})$. Then  $\mathscr{D}_\gamma=\mathscr{D}_\sigma$ as a consequence of the following two easily verifiable statements that we leave to the reader:
\begin{enumerate}
\item Let $(\theta_i)_{i \in I}$ be a family of maps in $Proj(A)$ and $\theta=\bigoplus_{i\in I} \theta_i$. Then $\mathscr{D}_{\theta}=\bigcap_{i\in I}\mathscr{D}_{\theta_i}$.
\item Let $\theta:Q_{-1}\rightarrow Q_0$ and $\beta:Q_{-1}\rightarrow  Q_0^\prime$ be maps in $Proj(A)$, and  $(\theta,\beta):Q_{-1}\rightarrow Q_0\oplus Q_0^\prime, \,p\mapsto(\theta(p),\beta(p))$. Now $\mathscr{D}_\theta\subseteq \mathscr{D}_{(\theta,\beta)}$.
\end{enumerate}
So $\bar{T}$ is a partial silting module as it lies in $\mathscr{D}_\gamma=\mathscr{D}_\sigma$, and it is even a silting module by Proposition \ref{eq def silting module}.
\end{proof}

\subsection{Examples}\label{examples subsection}
We have seen in Proposition \ref{tilting} that (partial) tilting modules are examples of (partial) silting modules. In this subsection we discuss  non-tilting examples of silting modules. An important class of examples of (partial) silting modules is given by $\tau$-rigid and support $\tau$-tilting modules over a finite dimensional $\Kbb$-algebra.

\begin{proposition}\label{tau-tilting}
Let $\Lambda$ be a finite dimensional $\Kbb$-algebra and let $T$ be in $mod(\Lambda)$. Then the following hold.
\begin{enumerate}
\item  $T$ is partial silting if and only if it is $\tau$-rigid.
\item $T$ is silting if and only if it is support $\tau$-tilting.
\item  \cite{Wei2}  $T$ is (finendo) quasitilting if and only if it is support $\tau$-tilting.  
\end{enumerate}
\end{proposition}

\begin{proof}
(1): This follows from Theorem \ref{tau-free description}(1) and (2) and Remark \ref{a few observations}(1) and (2).

(2) If $T$ is silting, then by (1) it is $\tau$-rigid, and it satisfies condition (S3) in Theorem \ref{eq def silting module},
where the $Gen(T)$-approximation sequence
$\xymatrix{\Lambda\ar[r]^{} & T_0\ar[r] & T_1\ar[r] & 0}$
 can be taken in $mod(\Lambda)$. Now the claim follows by \cite[Proposition 2.14]{J}.
The converse implication follows from Theorem \ref{tau-free description}(3).

(3) First, recall that finitely generated $\Lambda$-modules are always finendo.   By (2) the statement can be rephrased by saying that $T$ is quasitilting if and only if it is silting. Now the if-part is just  Proposition \ref{quasi}(1).    We show that for $T$ in $mod(\Lambda)$ also the converse  holds true. If $T$ is quasitilting, then by Proposition \ref{bar}, it satisfies condition (S3) in Theorem \ref{eq def silting module},
and it  is Ext-projective in $Gen(T)$. By Theorem \ref{tau-free description}(2) the latter  means that $T$ is $\tau$-rigid. We conclude from (1) that $T$ is a partial silting module satisfying (S3), or equivalently, a silting module. 
\end{proof}

\begin{remark}\label{functfin}
(1) Corollary \ref{air} can now be viewed as an analog of \cite[Theorem 2.7]{AIR} stating that over a finite dimensional algebra  $\Lambda$, there is a bijection between isomorphism classes of basic support $\tau$-tilting modules and functorially finite torsion classes $\Tcal$ in $mod(\Lambda)$. Indeed,  left $\Tcal$-approximations in $mod(\Lambda)$ can be chosen to be minimal, and then the cokernel is always  Ext-projective by a well-known lemma due to Wakamatsu.
 
 \smallskip
 
(2) A further consequence of  Proposition \ref{tau-tilting} is that for any  support $\tau$-tilting module $T$ over a finite dimensional $\Kbb$-algebra $\Lambda$, the functor $Hom_\Lambda(T,-)$ induces an equivalence between $Gen(T_A)$ and $Cogen(\mathsf{D}(T)_B)$, 
cf.~\cite[Proposition 3.5]{J} and \cite[Theorem 4.4]{HKM}. For more details on such equivalences see \cite{CDT}.
\end{remark}

The following is an example (taken from \cite[Example 5.3]{CDT}) of a finitely generated silting module which is neither tilting nor finitely presented.

\begin{example}\label{CDT}
Let  $Q$ be a quiver with two vertices, 1 and 2, and countably many arrows from 1 to 2. Let $P_i$ be the indecomposable projective $\Kbb Q$-module $e_i\Kbb Q$ for $i=1,2$. We show that $T:=P_2/soc(P_2)$ is a silting module (of projective dimension one) which is not tilting. 
Indeed, as observed in \cite{CDT}, the class $Gen(T)$ consists precisely of the semisimple injective $\Kbb Q$-modules and, thus, we have $Gen(T)=(P_1)^{\circ}\subsetneq T^{\perp_1}$. In particular, $T$ is not a tilting module. Of course,  $T$ is not finitely presented. It admits the following projective presentation
$$\xymatrix{0\ar[r] & P_1^{(\Nbb)}\ar[r]^\sigma & P_2\ar[r] & T\ar[r] & 0,}$$
with $\mathscr{D}_\sigma=T^{\perp_1}$. Let $\gamma$ be the projective presentation of $T$ obtained as the direct sum of $\sigma$ with the trivial map $P_1\ra 0$. 
Then we have that
 $$\mathscr{D}_{\gamma}=T^{\perp_1}\cap P_1\,^{\circ}=P_1\,^{\circ}=Gen(T),$$
thus proving that $T$ is a silting module.
\end{example}

\section{Silting complexes}\label{Silting complexes}
In this section we discuss (large) silting complexes and how they relate to t-structures, co-t-structures and silting modules. 
We first investigate the bijections between silting complexes and certain t-structures and co-t-structures (\cite{KY}). Then we show that mapping a 2-silting complex to its cohomology defines a bijection between (equivalence classes of) 2-silting complexes and (equivalence classes of) silting modules. In particular, this justifies our choice of name for the class of modules under study.

\subsection{Silting complexes, t-structures and co-t-structures}
We begin by extending the notion of silting and presilting complexes in order to include complexes of large projective modules. We adopt a definition due to Wei 
\cite[Definition 3.1]{Wei1}, who called such complexes semi-tilting.

\begin{definition}\label{big silting}
A bounded complex of projective $A$-modules $\sigma$ is said to be \textbf{presilting} if 
\begin{enumerate}
\item $Hom_{D(A)}(\sigma,\sigma^{(I)}[i])=0$, for all sets $I$ and $i>0$.
\end{enumerate}
It is furthermore \textbf{silting} if it also satisfies
\begin{enumerate}
\item[(2)] the smallest triangulated subcategory of $D(A)$ containing $Add(\sigma)$ is $K^b(Proj(A))$.
\end{enumerate}
We  call $\sigma$  \textbf{$n$-presilting}, respectively \textbf{$n$-silting}, if it is an $n$-term complex of projective $A$-modules. Hereby, and throughout this section, an \textbf{$n$-term} complex of projective modules means a complex concentrated between degrees $-n+1$ and $0$.
\end{definition}

For a presilting complex $\sigma$, we investigate the subcategory $aisle(\sigma)$ from Example \ref{example t-str}(4),  and the subcategory 
$\sigma^{\perp_{>0}}$. They  play an important role in determining whether $\sigma$ is silting or not, cf.~\cite[Theorem 1.3]{HKM} and \cite[Corollary 4.7]{AI}.

\begin{proposition}\label{silting criterion}
The following statements are equivalent for an $n$-term complex $\sigma$ of projective $A$-modules.
\begin{enumerate}
\item The complex $\sigma$ is ($n$-)silting. 
\item $\sigma$ is presilting, $\sigma^{\perp_{>0}}\cap D^{\leq 0}$ is closed for coproducts in $D(A)$, and the set $\{\sigma[i]:i\in\Zbb\}$ generates $D(A)$.
\item $aisle(\sigma)=\sigma^{\perp_{>0}}$.
\item $\sigma$ is presilting and $\sigma^{\perp_{>0}}$ lies in $D^{\leq 0}$.
\end{enumerate}
\end{proposition}
\begin{proof}
(1)$\Rightarrow$(2): It follows from \cite[Proposition 4.2]{Wei1} that $\sigma^{\perp_{>0}}$ is closed for coproducts in $D(A)$. By definition, the smallest triangulated subcategory of $D(A)$ containing $Add(\sigma)$ contains $A$. Then the smallest triangulated subcategory of $D(A)$ closed under coproducts and containing  $\sigma$ is $D(A)$. It then follows from Remark \ref{generating conditions} that $D(A)$ is generated by $\{\sigma[i],i\in\mathbb{Z}\}$.

(2)$\Rightarrow$(3): The arguments are similar to those in the proof of \cite[Corollary 4.7]{AI}. The subcategory $\sigma^{\perp_{>0}}\cap D^{\leq 0}$ is suspended and, by assumption, closed for coproducts in $D(A)$. It follows from  Example \ref{example t-str}(4) that  $\sigma^{\perp_{>0}}\cap D^{\leq 0}$   contains  $aisle(\sigma)$.
For any $X$ in $\sigma^{\perp_{>0}}$, there is a triangle associated with the t-structure $(aisle(\sigma),\sigma^{\perp_{<0}})$
$$Y\rightarrow X\rightarrow Z\rightarrow Y[1],$$
with $Y$ in $aisle(\sigma)$ and $Z$ in $\sigma^{\perp_{\leq 0}}$. Since $Y$ then also lies in $ \sigma^{\perp_{>0}}$, and $X$ lies in $\sigma^{\perp_{>0}}$  by assumption, we conclude that $Z$ lies in $\sigma^{\perp_{>0}}$. But then $Z$ lies in $\sigma^{\perp_{\leq 0}}\cap\sigma^{\perp_{>0}}$, and so $Z=0$ by (2).  It follows that  $\sigma^{\perp_{>0}}=aisle(\sigma)$.

(3)$\Rightarrow$(4): Since $\sigma$ lies in $D^{\leq 0}$, then so does $aisle(\sigma)=\sigma^{\perp_{>0}}$.

(4)$\Rightarrow$(1): This follows from  \cite[Proposition 3.12]{Wei1}.
\end{proof}

\begin{remark}\label{compatible} Definitions \ref{small silting} and \ref{big silting} agree on complexes $\sigma\in K^b(proj(A))$. Indeed, for condition (1)  one uses that $\sigma$ is a compact object. This also implies that $\sigma^{\perp_{>0}}$ is closed for coproducts. The claim now follows from (1)$\Leftrightarrow$(2) in Proposition \ref{silting criterion}.
\end{remark}

We will now study the (co-)t-structures arising from silting complexes in some more detail.

\begin{definition}\label{silting}
\begin{enumerate}
\item A t-structure $(\Vcal^{\leq 0},\Vcal^{\geq 0})$ (respectively, a co-t-structure $(\Ucal_{\geq 0},\Ucal_{\leq 0})$) in $D(A)$ is said to be \textbf{intermediate} if there are $a,b\in\mathbb{Z}$, $a\leq b$, such that $D^{\leq a}\subseteq \Vcal^{\leq 0}\subseteq D^{\leq b} \: (\text{respectively, }\, D^{\leq a}\subseteq \Ucal_{\leq 0}\subseteq D^{\leq b}).$
\item A t-structure $(\Vcal^{\leq 0},\Vcal^{\geq 0})$ is said to be \textbf{silting} if it is intermediate and there is a silting complex $\sigma$ in $D(A)$ such that  $\Vcal^{\leq 0}\cap {}^{\perp_0}(\Vcal^{\leq 0}[1]) = Add(\sigma).$ It is furthermore said to be \textbf{$n$-silting} if $D^{\leq -n+1}\subseteq \Vcal^{\leq 0}\subseteq D^{\leq 0}$.
\end{enumerate}
\end{definition}

\begin{lemma}\label{silting t-str}
A t-structure $(\Vcal^{\leq 0},\Vcal^{\geq 0})$ is $n$-silting with $\Vcal^{\leq 0}\cap {}^{\perp_0}(\Vcal^{\leq 0}[1]) = Add(\sigma)$ if and only if $\sigma$ is an $n$-silting complex and $\Vcal^{\leq 0}=\sigma^{\perp_{>0}}$. 
\end{lemma}

\begin{proof}
Suppose that $(\Vcal^{\leq 0},\Vcal^{\geq 0})$ is an $n$-silting t-structure with $\Vcal^{\leq 0}\cap {}^{\perp_0}(\Vcal^{\leq 0}[1]) = Add(\sigma)$. It is clear that $\sigma$ is a silting complex (since it has the same additive closure as a silting complex). From Proposition \ref{silting criterion} we have that $\sigma^{\perp_{>0}}=aisle(\sigma)$ and, hence, $\sigma^{\perp_{>0}}$ is contained in $\Vcal^{\leq 0}$ as so is $\sigma$. It remains to see that $\Vcal^{\leq 0}\subseteq\sigma^{\perp_{>0}}$. 
By the orthogonality relations of  t-structures, it is enough to prove that $\sigma^{\perp_{<0}}$ is contained in $\Vcal^{\geq 0}$. Let $X$ lie in $\sigma^{\perp_{<0}}$ and consider the canonical triangle associated with the t-structure $(\Vcal^{\leq 0},\Vcal^{\geq 0})$
$$Y\rightarrow X\rightarrow Z\rightarrow Y[1],$$
where $Y$ lies in $\Vcal^{\leq -1}$ and $Z$ lies in $\Vcal^{\geq 0}$. By assumption, we have that $Hom_{D(A)}(\sigma,X[i])=0$ for all $i<0$ and, since $\sigma$ lies in $\Vcal^{\leq 0}$, we also have that $Hom_{D(A)}(\sigma,Z[i])=0$ for all $i<0$. Thus, we have that $Hom_{D(A)}(\sigma,Y[i])=0$ for all $i<0$. On the other hand, since $\sigma$ lies in ${}^{\perp_0}(\Vcal^{\leq 0}[1])={}^{\perp_0}(\Vcal^{\leq -1})$ we see that $Hom_{D(A)}(\sigma,Y[i])=0$ for all $i\geq 0$. Recalling from Proposition \ref{silting criterion} that $\{\sigma[i]:i\in\Zbb\}$ is a set of generators for $D(A)$, we conclude that $Y=0$. Thus $X\cong Z$ and $X$ lies in $\Vcal^{\geq 0}$ as wanted. 

It remains to show that $\sigma$ is an $n$-term complex. Let $\sigma$ be a complex of projective $A$-modules of the form $(P_i,d_i)_{i\in\Zbb}$ with $P_i=0$ for all $i>0$. Since the t-structure is $n$-silting, $D^{\leq -n+1}$ lies in $\Vcal^{\leq 0}$ and  $\sigma$ lies in ${}^{\perp_0}(\Vcal^{\leq 0}[1])$, so $Hom_{D(A)}(\sigma,D^{\leq -n})=0$. Consider now the canonical co-t-structure $(K_{\geq 0},K_{\leq 0})$ in $K_p(A)$ from Example \ref{example t-str}(2), and take a   triangle given by  stupid truncations, that is, 
$$\xymatrix{Y \ar[r] & \sigma \ar[r]^{u\ \ } & Z\ar[r] & Y[1],}$$
with $Y$ in $K_{\geq -n+1}\cap K_{\leq 0}$ and $Z$ in $K_{\leq -n}$. Since $Z$ lies in $D^{\leq -n}$, the  map $u$ is zero and, thus, $\sigma$ lies in $K_{\geq -n+1}\cap K_{\leq 0}$ because it is a summand of $Y$ (in fact, it is easy to check that we even have $Y\cong \sigma$).

Conversely, let $\sigma$ be an $n$-silting complex. Then it is easy to see that $D^{\leq -n+1}\subseteq \sigma^{\perp_{>0}}$, and from Proposition \ref{silting criterion} we have that $\sigma^{\perp_{>0}}\subseteq D^{\leq 0}$. Moreover, clearly we have $Add(\sigma)\subseteq \sigma^{\perp_{>0}}\cap {}^{\perp_0}(\sigma^{\perp_{>0}}[1])$. We now show the reverse inclusion. Let $X$ lie in $\sigma^{\perp_{>0}}\cap {}^{\perp_0}(\sigma^{\perp_{>0}}[1])$ and let $I$ be the set $Hom_{D(A)}(\sigma, X)$. The canonical universal map $u:\sigma^{(I)}\rightarrow X$ gives rise to a triangle
$$\xymatrix{K\ar[r]& \sigma^{(I)}\ar[r]^u& X\ar[r]^v& K[1].}$$
Applying the functor $Hom_{D(A)}(\sigma,-)$ to the triangle, since $Hom_{D(A)}(\sigma,\sigma^{(I)}[1])=0$ and $Hom_{D(A)}(\sigma, u)$ is surjective, 
we deduce that $Hom_{D(A)}(\sigma,K[1])=0$. For $i>0$, since 
$$Hom_{D(A)}(\sigma, \sigma^{(I)}[i+1])=0=Hom_{D(A)}(\sigma,X[i])$$
we also conclude that $Hom_{D(A)}(\sigma,K[i+1])=0$. Thus $K$ lies in $\sigma^{\perp_{>0}}$. Since $X$ lies in ${}^{\perp_0}(\sigma^{\perp_{>0}}[1])$, we infer that  $v\in Hom_{D(A)}(X,K[1])$  is zero. Therefore, $u$ splits and $X$ lies in $Add(\sigma)$ as wanted.
Thus $(\sigma^{\perp_{>0}},\sigma^{\perp_{<0}})$ is an $n$-silting t-structure.
\end{proof}

It follows from the lemma that two silting complexes $\sigma$ and $\gamma$ in $D(A)$ satisfy $Add(\sigma)=Add(\gamma)$ if and only if $\sigma^{\perp_{>0}}=\gamma^{\perp_{>0}}$. Therefore we can define, unambiguously, a notion of equivalence of silting complexes: two silting complexes $\sigma$ and $\gamma$ are said to be \textbf{equivalent} if $Add(\sigma)=Add(\gamma)$.  

\smallskip

The following theorem generalises the correspondence of (compact) silting complexes with t-structures and co-t-structures  in \cite{AI,KY}. It has been partly treated in \cite[Theorem 5.3]{Wei1} and in \cite[Corollary 5.9]{MSSS}.

\begin{theorem}\label{bijections general}
There are bijections between
\begin{enumerate}
\item equivalence classes of silting complexes in $D(A)$;
\item silting t-structures in $D(A)$;
\item intermediate co-t-structures $(\Ucal_{\geq 0},\Ucal_{\leq 0})$ in $D(A)$ with $\Ucal_{\leq 0}$  closed for coproducts in $D(A)$;
\item \cite[Corollary 5.9]{MSSS} bounded co-t-structures in $K^b(Proj(A))$.
\end{enumerate}
\end{theorem}
\begin{proof}
Consider the following assignments.
$$\begin{array}{ccccl}
\text{Bijection} &&&& \text{Assignment}\\
\hline
 (1)\rightarrow (2) &&&& \Psi:\;\sigma\,\mapsto\; (\sigma^{\perp_{>0}},\sigma^{\perp_{<0}})\\
 (2)\rightarrow (1) &&&& \Theta:\;(\Vcal^{\leq 0},\Vcal^{\geq 0})\,\mapsto\; \sigma\text{ with } Add(\sigma)=\Vcal^{\leq 0}\cap {}^{\perp_0}(\Vcal^{\leq 0}[1])\\
 (1)\rightarrow (3) &&&&  \Phi:\;\sigma\,\mapsto\; ({}^{\perp_0}(\sigma^{\perp_{>0}}[1]),\sigma^{\perp_{>0}})\\
 (1)\rightarrow (4) &&&& \Omega :\;\sigma\,\mapsto\; ({}^{\perp_0}(\sigma^{\perp_{>0}}[1])\cap K^b(Proj(A)),\sigma^{\perp_{>0}}\cap K^b(Proj(A)))
\end{array}$$

We have seen above that the assignments $\Psi$, $\Phi$ and $\Omega$ do not depend on the representative of the equivalence class of the silting complex $\sigma$. Note also that these assignments commute with the shift functor [1], which is an auto-equivalence of the derived category. To show that $\Psi$, $\Phi$ and $\Theta$ are bijections, we will assume without loss of generality that silting complexes are concentrated in degrees less or equal than $0$ or that $\sigma^{\perp_{>0}}$ is contained in $D^{\leq 0}$. 
It follows immediately from Lemma \ref{silting t-str} that the assignments $\Psi$ and $\Theta$ are inverse to each other.

We prove that $\Phi$ is a bijection in two steps. The first step will provide a bijection between (1) and certain co-t-structures in $D^-(A)$, and the second step will relate them to the co-t-structures in (2).

\underline{Step 1:}
In \cite[Theorem 5.3]{Wei1} it is shown that assigning to a silting complex $\sigma$ the subcategory $\sigma^{\perp_{>0}}$ yields a bijection between  equivalence classes of silting complexes and subcategories $\mathcal U$ of $D^-(A)$ satisfying four properties. 
We leave to the reader to check that two of those properties (being specially covariantly finite and coresolving, as defined in \cite{Wei1}) correspond exactly to the statement that $({}^{\perp_0}(\mathcal U[1]),\mathcal U)$ is a co-t-structure in $D^-(A)$. Notice that here the left orthogonal is computed in $D^-(A)$. 
A third property states that $\mathcal U$ is closed for coproducts. 

We turn to the fourth property. It asserts that  
every object $X$ in $D^-(A)$ admits a finite coresolution  by $\Ucal$, i.e. there are a positive integer $m$, a collection of objects $(U_i)_{0\leq i\leq m}$ in $\Ucal$, and a finite sequence of triangles as follows
$$\begin{array}{c} X\rightarrow U_0\rightarrow C_0\rightarrow X[1]\\ C_0\rightarrow U_1\rightarrow C_1\rightarrow C_0[1]\\ ... \\ C_{m-2}\rightarrow U_{m-1}\rightarrow U_{m}\rightarrow C_m[1]. 
\end{array}$$
We now prove that this property can be rephrased by saying that the co-t-structure $({}^{\perp_0}(\mathcal U[1]),\mathcal U)$ in $D^-(A)$ is  intermediate. In fact, the classes $\Ucal$ occurring in \cite[Theorem 5.3]{Wei1} satisfy this condition by 
 \cite[Lemma 4.1]{Wei1}.
Conversely,  given a co-t-structure  $(\Ucal_{\geq 0},\Ucal_{\leq 0})$  in $D^-(A)$ such that $D^{\leq -n}\subseteq \Ucal_{\leq 0}\subseteq D^{\leq 0}$ for some  $n$,  we 
take  a complex $X$ in $D^-(A)$, say with $H^i(X)=0$ for all $i>k$, and construct a sequence of triangles as above. Let us first reduce this analysis to the case where $X$ lies in $D^{\leq 0}$. Indeed, by the axioms of co-t-structure we have a triangle
\begin{equation}\nonumber X\rightarrow U_0\rightarrow C_0\rightarrow X[1]
\end{equation}
such that $U_0$ lies in $\Ucal_{\leq 0}$ and $C_0$ lies in $\Ucal_{\geq 0}$. Using that $\Ucal_{\leq 0}\subseteq D^{\leq 0}$, we see that $H^i(C_0)=0$ for all $i>k-1$. So we can find a finite sequence of triangles yielding an object $C_{k-1}$ in $D^{\leq 0}$. Therefore, without loss of generality, we may assume that we start with $X$ in $D^{\leq 0}$. 
We now build a sequence of triangles
$$\begin{array}{c} X\rightarrow U_0\rightarrow C_0\rightarrow X[1]\\ C_0\rightarrow U_1\rightarrow C_1\rightarrow C_0[1]\\ ... \\ C_{n-1}\rightarrow U_{n}\rightarrow C_{n}\rightarrow C_{n-1}[1] \end{array}$$
where $U_i$ lies in $\Ucal_{\leq 0}$ and $C_i$ lies in $\Ucal_{\geq 0}$ for all $0\leq i\leq n$. Here $n$ is the natural number above with $D^{\leq -n}\subseteq \Ucal_{\leq 0}$.
We claim that  $Hom_{D(A)}(C_n,C_{n-1}[1])=0$. This will show that the last triangle splits, so $C_{n-1}$ will belong to  $\Ucal_{\leq 0}$ as wished.
To prove this claim, we apply the functor $Hom_{D(A)}(C_n,-)$ to all the triangles. From the orthogonality properties of the co-t-structure we infer that, for all $1\leq i\leq n$,
$$Hom_{D(A)}(C_n, C_{n-1}[1])\cong Hom_{D(A)}(C_n, C_{n-i}[i]).$$
From the first triangle we get an isomorphism $$Hom_{D(A)}(C_n,C_0[n])\cong Hom_{D(A)}(C_n,X[n+1]).$$ But $C_n$ lies in $\Ucal_{\geq 0}={}^{\perp_0}(\Ucal_{\leq 0}[1])$, and $X[n+1]$ lies in $D^{\leq -n-1}=D^{\leq -n}[1]\subseteq \Ucal_{\leq 0}[1]$. So we conclude that $Hom_{D(A)}(C_n,X[n+1])=0$, which proves our claim.

\underline{Step 2.}
We have shown that $\Phi$ defines a bijection between equivalence classes of silting complexes in $D(A)$ and intermediate co-t-structures $(\Ucal_{\geq 0},\Ucal_{\leq 0})$ in $D^-(A)$ such that $\Ucal_{\leq 0}$ is closed for coproducts in $D(A)$.
 It remains to prove that such co-t-structures in $D^-(A)$  and  the corresponding co-t-structures in $D(A)$ are in bijection.
To this end, we prove that for such a co-t-structure $(\Ucal_{\geq 0},\Ucal_{\leq 0})$ in $D^-(A)$, the pair 
$({}^{\perp_0}(\Ucal_{\leq 0}[1]), \Ucal_{\leq 0})$ in $D(A)$ (now with the orthogonal computed in $D(A)$) is an intermediate co-t-structure in $D(A)$. Then we immediately obtain an injective assignment with an obvious inverse given by the intersection with $D^-(A)$,  completing our proof. 

We only have to verify axiom (3) in the definition (\ref{t-def}) of a co-t-structure for the pair $({}^{\perp_0}(\Ucal_{\leq 0}[1]), \Ucal_{\leq 0})$ in $D(A)$. We use the equivalence between  $D(A)$ and $K_p(A)$ and consider the standard co-t-structure 
$(K_{\geq 0},K_{\leq 0})$  in $K_p(A)$ from Example \ref{example t-str}(2). For any $X$ in $K_p(A)$, using stupid truncation, there is a triangle 
$$\xymatrix{Y\ar[r]&X\ar[r]^\psi&Z\ar[r]&Y[1]}$$
where $Y$ in $K_{\geq 1}$ and $Z$ in $K_{\leq 0}$.
Now, $Z$ lies in $D^-(A)$ and, thus, there is a triangle
$$\xymatrix{C[-1]\ar[r]& Z\ar[r]^\theta& U\ar[r]& C}$$
with $U$ in $\Ucal_{\leq 0}$ and $C$ in $\Ucal_{\geq 0}\subset{}^{\perp_0}(\Ucal_{\leq 0}[1])$. Using the octahedral axiom, one can check that there is a triangle
$$\xymatrix{Y[1]\ar[r]& Cone(\theta\psi)\ar[r] & C\ar[r]& Y[2].}$$
Since $Y[1]$ lies in $K_{\geq 0}$ and homotopically projective resolutions of complexes in $\Ucal_{\leq 0}[1]$ lie in $K_{\leq -1}$, we have that $Y[1]$ lies in ${}^{\perp_0}(\Ucal_{\leq 0}[1])$. Since $C$ also lies in ${}^{\perp_0}(\Ucal_{\leq 0}[1])$,  so does $Cone(\theta\psi)$, thus yielding a co-t-structure triangle
$$\xymatrix{Cone(\theta\psi)[-1]\ar[r]& X\ar[r]^{\theta\psi}& U\ar[r]& Cone(\theta\psi)}$$
with $Cone(\theta\psi)$ in ${}^{\perp_0}(\Ucal_{\leq 0}[1])$ and $U$ in $\Ucal_{\leq 0}$, as wanted.

It remains to see that $\Omega$ is a bijection. From \cite[Corollary 5.9]{MSSS} there is a bijection between equivalence classes of silting complexes in $D(A)$ and bounded co-t-structures in $K^b(Proj(A))$. It associates to a silting complex $\sigma$  the pair $({}^{\perp_0}(\Ucal_\sigma[1]),\Ucal_\sigma)$ in $K^b(Proj(A))$, where $\Ucal_\sigma$ is the smallest suspended subcategory of $K^b(Proj(A))$ containing $Add(\sigma)$. Its inverse is given by considering an additive generator of the intersection of the pair of subcategories. We only need to check that $\Ucal_\sigma=\sigma^{\perp_{>0}}\cap K^b(Proj(A))$. First we observe that the intermediate co-t-structure $({}^{\perp_0}(\sigma^{\perp_{>0}}[1]),\sigma^{\perp_{>0}})$ in $D(A)\cong K_p(A)$ restricts to $K^b(Proj(A))$. Indeed, for a complex  $X$ in $K^b(Proj(A))$, consider a co-t-structure triangle in $K_p(A)$ with respect to $({}^{\perp_0}(\sigma^{\perp_{>0}}[1]),\sigma^{\perp_{>0}})$, say
$$C[-1]\rightarrow X\rightarrow U\rightarrow C,$$
with $U$ in $\sigma^{\perp_{>0}}$ and $C$ in ${}^{\perp_0}(\sigma^{\perp_{>0}}[1])$. Then $U$ is a right bounded complex in $K_p(A)$ and one can easily check that $C$ is left bounded in $K_p(A)$. Therefore, they are all complexes in $K^b(Proj(A))$. Clearly, this restricted co-t-structure corresponds to the complex $\sigma$ under the bijection defined in \cite[Corollary 5.9]{MSSS} and, thus, $\Ucal_\sigma$ and $\sigma^{\perp_{>0}}\cap K^b(Proj(A))$ must coincide.
\end{proof}

\begin{remark}
(1)  We clarify in some more detail how the result above generalises Theorem \ref{KY correspondence} for compact silting complexes over a finite dimensional $\Kbb$-algebra $\Lambda$. In \cite{KY}, it is shown that if $\sigma$ is compact, then $(\sigma^{\perp_{>0}}\cap D^b(mod(\Lambda)),\sigma^{\perp_{<0}}\cap D^b(mod(\Lambda)))$ is a t-structure in $D^b(mod(\Lambda))$. Adopting the notation in  the proof of Theorem \ref{bijections general}, this corresponds to the restriction of the t-structure $\Psi(\sigma)$  in $D(\Lambda)$ to $D^b(\Lambda)$. Indeed, it can be checked that such a restriction must be bounded (because $\Psi(\sigma)$ is intermediate) and that the heart is a module category (the zero cohomology of $\sigma$ with respect to the t-structure is a projective generator of the heart and it is small because $\sigma$ is compact). Moreover, the co-t-structure associated to $\sigma$ under Theorem \ref{KY correspondence} can be checked (using the description provided in \cite{MSSS}) to coincide with the restriction  of $\Omega(\sigma)$ to $K^b(proj(\Lambda))$. It is, therefore, also the restriction of $\Phi(\sigma)$ to $K^b(proj(\Lambda))$. Again, this restriction will be bounded because the co-t-structure is intermediate.

\smallskip

(2) Notice that $\Psi(\sigma)$ is a t-structure which is always right adjacent to the co-t-structure $\Phi(\sigma)$, see  \cite[Definition 4.4.1]{Bo}. Compare with \cite{KY}.
\end{remark}

\subsection{2-silting complexes and silting modules}

The following lemma establishes a useful connection between $\sigma^{\perp_{>0}}$ and the torsion class $\mathscr{D}_\sigma$ introduced in Subection \ref{subsection silting}, for any 2-term complex $\sigma$ in $K^b(Proj(A))$.

\begin{lemma}\label{lie in orth}
The following hold for a 2-term complex $\sigma$ in $K^b(Proj(A))$ with $T=H^0(\sigma)$.  
\begin{enumerate}
\item An object $X$  in $D^{\leq 0}$ belongs to $\sigma^{\perp_{>0}}$ if and only if $H^0(X)$ lies in $\mathscr{D}_\sigma$. Moreover, $\mathscr{D}_\sigma=\sigma^{\perp_{>0}}\cap Mod(A)$.
\item An object $X$  in $D^{\geq 0}$ belongs to  $\sigma^{\perp_{\leq 0}}$ if and only if $H^0(X)$ lies in $T^\circ$. Moreover, $T^\circ=\sigma^{\perp_{\leq 0}}\cap Mod(A).$ 
\item  The module $T$ is partial silting with respect to $\sigma$ if and only if the complex $\sigma$ is presilting and $\sigma^{\perp_{>0}}\cap D^{\leq 0}$ is closed for coproducts in $D(A)$.
\end{enumerate}
\end{lemma}
\begin{proof}
We set $\sigma:P_{-1}\rightarrow P_0$. 

(1) Let $X=(X_j,d_j)_{j\in\Zbb}$ be a complex in $D^{\leq 0}$ (assume without loss of generality that $X_j=0$ for all $j>0$). Suppose that $X$ lies in $\sigma^{\perp_{>0}}$. Any map from $h:P_{-1}\rightarrow H^0(X)$ lifts to a map $f:P_{-1}\rightarrow X_0$ via the projection map $\pi:X_0\rightarrow H^0(X)$ since $P_{-1}$ is projective. Now, $f$ induces a map in $Hom_{K(A)}(\sigma,X[1])$ which we assume to be zero. Thus, there are maps $s_{0}:P_0\rightarrow X_0$ and $s_{-1}:P_{-1}\rightarrow X_{-1}$ such that $f=s_0\sigma+d_{{-1}}s_{-1}$. Since $h=\pi f$, we easily see that $h=(\pi s_0)\sigma$ and, thus, $H^0(X)$ lies in $\mathscr{D}_\sigma$.

Conversely, suppose that $H^0(X)$ lies in $\mathscr{D}_\sigma$. Then, for a morphism in $Hom_{K(A)}(\sigma,X[1])$ defined by a map $f:P_{-1}\rightarrow X_0$, there is $h:P_0\rightarrow H^0(X)$ such that $\pi f = h\sigma$. Since $P_0$ is projective,  there is $s_0:P_0\rightarrow X_0$ such that $\pi s_0=h$. It is then easy to observe that  there is $s_{-1}:P_{-1}\rightarrow X_{-1}$ such that $f-s_0\sigma=d_{-1}s_{-1}$, showing that $f$ is null-homotopic.

(2) Let $X$ be an object in $\sigma^{\perp_{\leq 0}}\cap D^{\geq 0}$. Since $X$ lies in $D^{\geq 0}$, we have a (standard) t-structure triangle of the form
\begin{equation}\nonumber
(\tau^{\geq 1}X)[-1]\rightarrow H^0(X) \rightarrow X \rightarrow \tau^{\geq 1}X.
\end{equation}
Since $Hom_{D(A)}(\sigma,(\tau^{\geq 1}X)[-1])=0=Hom_{D(A)}(\sigma,X)$, we get that $Hom_{D(A)}(\sigma,H^0(X))=0$ and, thus, $H^0(X)$ lies in $T^\circ$. Similarly one proves the converse. 

(3) First we claim that $\sigma^{\perp_{>0}}\cap D^{\leq 0}$ is closed for coproducts if and only if $\mathscr{D}_\sigma$ is closed for coproducts, i.e. condition (S1) in the definition of partial silting module holds for $T$. Indeed, consider the canonical triangle $$\tau^{\leq -1}\bigoplus_{i\in I}X_i\longrightarrow \bigoplus_{i\in I}X_i\longrightarrow H^0(\bigoplus_{i\in I}X_i)\longrightarrow (\tau^{\leq -1}\bigoplus_{i\in I}X_i)[1],$$ for any family of objects $(X_i)_{i\in I}$ in $\sigma^{\perp_{>0}}\cap D^{\leq 0}$. Since $D^{\leq -1}$ is contained in $\sigma^{\perp_{>0}}\cap D^{\leq 0}$ and $H^0$ commutes with coproducts, our claim follows from (1).
Condition (S2) is equivalent to $\sigma$ lying in $\sigma^{\perp_{>0}}$ by Lemma \ref{ext-orthogonality1}(3).
\end{proof}

The following theorem  is a non-compact version
of \cite[Theorem 2.10]{HKM}, in the sense that it extends the statements from compact silting complexes to silting complexes in $K^b(Proj(A))$.

\begin{theorem}\label{eq def silting cpx}
Let $\sigma$ be 2-term complex in $K^b(Proj(A))$ and $T=H^0(\sigma)$. The following statements are equivalent.
\begin{enumerate}
\item $\sigma$ is a 2-silting complex;
\item $\sigma$ is a presilting complex, and $\{\sigma[i]:i\in\Zbb\}$ is a set of generators in $D(A)$;
\item $T$ is a silting module with respect to $\sigma$;
\item $(\mathscr{D}_\sigma, T^\circ)$ is a torsion pair in $Mod(A)$.
\end{enumerate}
Moreover, if the conditions above are satisfied, we have  $$\sigma^{\perp_{>0}}=D_{\mathscr{D}_\sigma}^{\leq 0}=\{X\in D(A):H^0(X)\in\mathscr{D}_\sigma,\; H^i(X)=0 \;\forall i>0\}.$$
\end{theorem}
\begin{proof}
(1)$\Rightarrow$(2): This follows from Proposition \ref{silting criterion}.

(2)$\Rightarrow$(1): By Proposition \ref{silting criterion}, we have to  show that $\sigma^{\perp_{>0}}$ is contained in $D^{\leq 0}$.
Let $X$ be an object in $\sigma^{\perp_{>0}}$ and consider its triangle decomposition with respect to the canonical t-structure in $D(A)$
$$\tau^{\leq 0}X\rightarrow X\rightarrow \tau^{\geq 1}X\rightarrow (\tau^{\leq 0}X)[1].$$
It is clear that  $Hom_{D(A)}(\sigma[i],\tau^{\geq 1}X)=0$ for $i\geq 0$. Moreover, applying $Hom_{D(A)}(\sigma[i],-)$, with $i<0$ to the triangle we get by assumption that $Hom_{D(A)}(\sigma,X[-i])=0$, and also that $Hom_{D(A)}(\sigma,(\tau^{\leq 0}X)[-i+1])=0$ since $(\tau^{\leq 0}X)[-i+1]$ lies in $D^{\leq -2}$ for all $i<0$. Therefore, we have that $Hom_{D(A)}(\sigma[i],\tau^{\geq 1}X)=0$ for all $i \in \Zbb$, and $\tau^{\geq 1}X=0$ as $\{\sigma[i]:i\in\Zbb\}$ is a set of generators for $D(A)$. 

(1)$\Rightarrow$(3): Combining Proposition \ref{silting criterion} with Lemma \ref{lie in orth}, we see that $T$ is partial silting, and so $Gen(T)\subseteq \mathscr{D}_\sigma$. Let now $M$ be a module in $\mathscr{D}_\sigma$ and take the universal map $u:\sigma^{(I)}\rightarrow M$, where $I=Hom_{D(A)}(\sigma, M)$. We will show that $H^0(u):T^{(I)}\rightarrow M$ is a surjection. For this purpose, we consider the triangle $$\sigma^{(I)}\stackrel{u}{\rightarrow} M \to C\to \sigma^{(I)}[1]$$ and prove that  $H^0(C)=0$. We use the generating property of the set $\{\sigma[i]:i\in\Zbb\}$, see (2) above. Note that the long exact sequence of cohomologies for the  triangle above shows that there is a surjection $M\rightarrow H^0(C)$. Hence, the module $H^0(C)$ lies in the torsion class $\mathscr{D}_\sigma$, and by Lemma \ref{lie in orth}, it lies also in $\sigma^{\perp_{>0}}$, that is, $Hom_{D(A)}(\sigma,H^0(C)[1])=0$. Since $\sigma$ is a two -term complex, it remains to see that $Hom_{D(A)}(\sigma,H^0(C))=0$. As $C$ lies in $D^{\leq 0}$, we have the following canonical triangle given by the standard t-structure in $D(A)$ 
$$\tau^{\leq -1}C\rightarrow C\rightarrow H^0(C)\rightarrow \tau^{\leq -1}C[1].$$
Now, on the one hand, since $\sigma$ is presilting, it follows from the definition of $C$ that $Hom_{D(A)}(\sigma, C)=0$. On the other hand, since $\sigma$ is a 2-term complex we also get that $Hom_{D(A)}(\sigma,\tau^{\leq -1}C[1])=0$. Therefore, we have $Hom_{D(A)}(\sigma,H^0(C))=0$, as wanted.

(3)$\Rightarrow$(4): This follows immediately from Remark \ref{a few observations}(1) as $\mathscr{D}_\sigma=Gen(T)$.

(4)$\Rightarrow$(1): Suppose that $(\mathscr{D}_\sigma, T^\circ)$ is a torsion pair. Then clearly $T$ is partial silting with respect to $\sigma$, 
which implies by Lemma \ref{lie in orth} that $\sigma$ is presilting and $\sigma^{\perp_{>0}}\cap D^{\leq 0}$ is closed for coproducts in $D(A)$. By Proposition \ref{silting criterion}, it remains to show that $\{\sigma[i]:i\in\mathbb{Z}\}$ generates $D(A)$. Let $X$ be an object of $D(A)$ such that $Hom_{D(A)}(\sigma,X[i])=0$ for all $i \in \Zbb$. Since $\sigma$ is concentrated in degrees $-1$ and $0$, this is equivalent to $Hom_{D(A)}(\sigma,\tau^{\leq 0}(X[i]))=0$ for all $i \in \Zbb$. Then $Hom_{D(A)}(\sigma,H^i(X))=0$, and thus  $H^i(X)$ lies in $T^\circ$ for all $i \in \Zbb$. Consider the triangle 
$$H^0(X[i-1])[-1]\rightarrow \tau^{\leq -1}(X[i-1])\rightarrow \tau^{\leq 0}(X[i-1])\rightarrow H^0(X[i-1])$$
and apply to it the functor $Hom_{D(A)}(\sigma,-)$. Since $$Hom_{D(A)}(\sigma,\tau^{\leq 0}(X[i-1]))=0=Hom_{D(A)}(\sigma,H^0(X[i-1])[-1]),$$ we conclude that
$$0=Hom_{D(A)}(\sigma,\tau^{\leq -1}(X[i-1]))=Hom_{D(A)}(\sigma,\tau^{\leq 0}(X[i])[1]),$$
thus showing that  $\tau^{\leq 0}(X[i])$ belongs to $\sigma^{\perp_{>0}}$ for all $i \in \Zbb$. By Lemma \ref{lie in orth}, it follows that $H^i(X)=H^0(\tau^{\leq 0}X[i])$ 
lies in $\mathscr{D}_\sigma$ for all $i \in \Zbb$. Since the pair $(\mathscr{D}_\sigma, T^\circ)$ is a torsion pair, we conclude that $H^i(X)=0$ for all $i \in \Zbb$, as wanted.

Let us now assume that the equivalent conditions (1)-(4) hold. In particular, $(\mathscr{D}_\sigma,T^\circ)$ is a torsion pair in $Mod(A)$ and so Example \ref{example t-str}(3) gives us a t-structure $(D^{\leq 0}_{\mathscr{D}_\sigma},D^{\geq 0}_{T^\circ})$. We want to prove that $\sigma^{\perp_{>0}}=D_{\mathscr{D}_\sigma}$. Proposition \ref{silting criterion} shows that $aisle(\sigma)=\sigma^{\perp_{>0}}\subseteq D^{\leq 0}$ and, thus, by Example \ref{example t-str}(4) and Lemma \ref{lie in orth}(1),
$$\sigma^{\perp_{>0}}=aisle(\sigma)\subseteq \{X\in D(A): H^0(X)\in \mathscr{D}_\sigma, H^i(X)=0, \;\forall i>0\}=D^{\leq 0}_{\mathscr{D}_\sigma}.$$
We will show that $aisle(\sigma)^{\perp_0}\subseteq D^{\geq 1}_{T^\circ}$, thus proving that the inclusion above is in fact an equality. Let $X$ be an object in  $aisle(\sigma)^{\perp_0}=\sigma^{\perp_{\leq 0}}$. It is clear that $Hom_{D(A)}(\sigma,(\tau^{\leq -1}X)[i])=0$ for all $i>0$. Consider now the triangle
\begin{equation}\nonumber
(\tau^{\geq 0}X)[i-1]\rightarrow (\tau^{\leq -1}X)[i] \rightarrow X[i]\rightarrow (\tau^{\geq 0}X)[i].
\end{equation}
Since $\sigma$ lies in $D^{\leq 0}$, we have that $Hom_{D(A)}(\sigma, (\tau^{\geq 0}X)[i-1])=0$ for all $i\leq 0$ and also, by the assumption on $X$, $Hom_{D(A)}(\sigma,X[i])=0$ for all $i\leq 0$. This shows that $Hom_{D(A)}(\sigma,(\tau^{\leq -1}X)[i])=0$ for all $i\leq 0$. Since $\{\sigma[i]:i\in\Zbb\}$ is a set of generators for $D(A)$, we conclude that $\tau^{\leq -1}X=0$. By Lemma \ref{lie in orth}(2), we get that $H^0(X)$ lies in $T^\circ$.
\end{proof}

\begin{remark}\label{silting equivalence}
(1) Theorem \ref{eq def silting cpx} shows that the t-structure generated by a silting complex $\sigma$ equals both the t-structure $(\sigma^{\perp_{>0}},\sigma^{\perp_{< 0}})$ studied by Hoshino-Kato-Miyachi in \cite{HKM} and the t-structure associated  to the torsion pair $(\mathscr{D}_\sigma,T^\circ)$ in the sense of  Happel-Reiten-Smal{\o} \cite{HRS}.

\smallskip

(2) We know from Theorem \ref{eq def silting cpx} that the cohomology $H^0(\sigma)$ is a silting module for any 2-silting complex $\sigma$. 
Further,  $\sigma$ and $\gamma$ are   
equivalent 2-silting complexes if and only if the silting modules $T=H^0(\sigma)$ and $T'=H^0(\gamma)$ are equivalent. 
Indeed, recall  that the silting modules $T$ and $T'$ are equivalent if $Add(T)=Add(T')$, which in turn means that $Gen(T)=Gen(T')$.
So the only-if-part follows from the fact that $H^0$ commutes with coproducts. Conversely, if $T$ and $T'$ are equivalent, then they generate the same torsion pair, and therefore the associated Happel-Reiten-Smal\o\ t-structures coincide, which means that $\sigma^{\perp_{>0}}=\gamma^{\perp_{>0}}$ by Theorem \ref{eq def silting cpx}.
\end{remark}

We finish by specialising Theorem \ref{bijections general} to 2-term complexes. For the bijection between (1) and (2) see also a related result in \cite{Wei3}. 

\begin{theorem}\label{bijections complexes}
There are bijections between
\begin{enumerate}
\item equivalence classes of 2-silting complexes;
\item equivalence classes of silting $A$-modules;
\item 2-silting t-structures in $D(A)$;
\item co-t-structures $(\Ucal_{\geq 0},\Ucal_{\leq 0})$ in $D(A)$ with $D^{\leq -1}\subseteq \Ucal_{\leq 0}\subseteq D^{\leq 0}$ and $\Ucal_{\leq 0}$  closed for coproducts in $D(A)$.
\end{enumerate}
\end{theorem}
\begin{proof}
Consider the following assignments.
$$\begin{array}{ccccl}
\text{Bijection} &&&& \text{Assignment}\\
\hline
 (1)\rightarrow (2) &&&&  H^0:\;\sigma\,\mapsto\; H^0(\sigma)\\
 (1)\rightarrow (3) &&&& \Psi:\;\sigma\,\mapsto\; (\sigma^{\perp_{>0}},\sigma^{\perp_{<0}})\\
 (1)\rightarrow (4) &&&& \Phi:\;\sigma\,\mapsto\; ({}^{\perp_0}(\sigma^{\perp_{>0}}[1]),\sigma^{\perp_{>0}})
\end{array}$$
Remark \ref{silting equivalence}(2) above
shows that $H^0$ is well-defined and injective. The surjectivity  follows from Theorem \ref{eq def silting cpx}, where it is shown that if $T$ is a silting module with respect to a projective presentation $\sigma$, then $\sigma$ is a 2-silting complex.
Moreover, it follows from Lemma \ref{silting t-str} that the map $\Psi$ from Theorem \ref{bijections general} induces  a bijection between equivalence classes of 2-silting complexes and 2-silting t-structures.
Finally the co-t-structure $({}^{\perp_0}(\sigma^{\perp_{>0}}[1]), \sigma^{\perp_{>0}})$ in $D(A)$ associated to a 2-silting complex $\sigma$ clearly satisfies $D^{\leq -1}\subseteq \sigma^{\perp_{>0}}\subseteq D^{\leq 0}$. The map $\Phi$ from Theorem \ref{bijections general}  therefore restricts to the stated bijection.
\end{proof}

\end{document}